\documentclass[reqno]{amsart} \usepackage[utf8]{inputenc}
\usepackage[UKenglish]{babel}
\usepackage{amsmath,amsfonts,amssymb,amsthm,enumitem,graphicx}
\usepackage[usenames,dvipsnames]{xcolor} \usepackage[width=\linewidth]{caption}
\usepackage{subcaption} 
\usepackage[colorlinks=true,linkcolor=blue]{hyperref} 
\newtheorem{theorem}{Theorem}[section] 
 
\theoremstyle{definition} \newtheorem{definition}[theorem]{Definition}

\newtheorem{remark}[theorem]{Remark} 
 \numberwithin{equation}{section}
  \newcommand{\Q}{\mathbb{Q}} \newcommand{\R}{\mathbb{R}}
 \newcommand{\A}{\mathbb{A}}
  
  \newcommand{\W}{\Omega}
\newcommand{\w}{\omega}  
\newcommand{\dd}{\textsf{d}} 
\newcommand{\wit}{\widetilde} 
\newcommand{\lsm}{\left[\!\begin{smallmatrix}}
    \newcommand{\rsm}{\end{smallmatrix}\!\right]} 
\newcommand{\des}{\displaystyle}

\DeclareMathOperator{\cls}{cls} 
\DeclareMathOperator{\Int}{Int} 
 
\definecolor{col}{rgb}{0,0,0.6}
\begin{document}
\title[Attractors in almost periodic Nicholson systems] {Attractors in almost
  periodic Nicholson systems and some numerical simulations}
\author[A.M. Sanz]{Ana M. Sanz} \author[V.M.~Villarragut]{V\'{\i}ctor
  M. Villarragut}
\address[A.M. Sanz]{Departamento de Did\'{a}ctica de las Ciencias
  Experimentales, Sociales y de la Matem\'{a}tica, Facultad de Educaci\'{o}n,
  Universidad de Valladolid, 34004 Palencia, Spain, and member of IMUVA,
  Instituto de Investigaci\'{o}n en Mate\-m\'{a}\-ti\-cas, Universidad de
  Valladolid.} \email{anamaria.sanz@uva.es}
\address[V.M.~Villarragut]{Departamento de Matem\'{a}tica Aplicada a la
  Ingenier\'{\i}a Industrial, Universidad Polit\'{e}cnica de Madrid, Calle de
  Jos\'{e} Guti\'{e}rrez Abascal 2, 28006 Madrid, Spain.}
\email{victor.munoz@upm.es} \thanks{The first author was partly supported by
  MICIIN/FEDER under project PID2021-125446NB-I00 and by Universidad de
  Valladolid under project PIP-TCESC-2020. The second author was supported by
  MICINN/FEDER under project PID2021-125446NB-I00.  } \date{}
\begin{abstract}
  The existence of a global attractor is proved for the skew-product semiflow
  induced by almost periodic Nicholson systems and new conditions are given for
  the existence of a unique almost periodic positive solution which
  exponentially attracts every other positive solution. Besides, some numerical
  simulations are included to illustrate our results in some concrete Nicholson
  systems.
\end{abstract}
\renewcommand{\subjclassname}{\textup{2020} Mathematics Subject Classification}
\maketitle
\section{Introduction}\label{sec-intro}
After Gurney et al.~\cite{gubl} came up with a scalar delay equation called
Nicholson's blowflies equation, there has been an increasing interest in the
dynamical behaviour of the solutions of this equation, as well as of its
generalisations, such as Nicholson systems or their non-autonomous versions,
often periodic or almost periodic.  Firstly, the interest is in the extinction
versus the persistence of the population, or of the population within some
patch, when dealing with a compartmental model. In the case of persistence, the
goal is to describe the picture of the population's evolution, depending on the
initial situation. These models lie within the field of delayed functional
differential equations (FDEs for
short). 
A relevant reference for the theory of FDEs is Hale and Verduyn Lunel~\cite{have}.

Although there is extensive literature on Nicholson models, we mention some
publications which are closer to our approach, namely, the works by
Faria~\cite{faria11,faria17,faria21}, Faria et al.~\cite{faos}, Faria and
R\"{o}st~\cite{faro}, and Obaya and Sanz~\cite{obsa16,obsa18}. The reader can find many
other references therein. In many papers dealing with the almost periodic model,
the idea is to impose conditions forcing an invariant zone where a unique almost
periodic solution is found by means of fixed-point theorems. New alternative
methods can be found in Zhang et al.~\cite{zhetal}. Very recently, almost periodic
Nicholson systems have been considered by Novo et al.~\cite{nosv}, as models in
biology where the use of the exponential ordering can lead to new conditions
that force the existence of a unique attracting almost periodic solution. This
turns out to be the case for persistent systems provided that the delays are
small enough.

In this paper, we apply methods of the theory of skew-product semiflows to
analyse the long-term dynamics of almost periodic Nicholson systems. The
monograph by Shen and Yi~\cite{shyi} can be a useful reference for this theory.
Due to the time dependence of the system, solutions do not define a semiflow in
a direct way. The almost periodic time variation in the model permits adding a
compact base flow component $\Omega$ by means of the so-called hull construction, so
that solutions induce a dynamical system on a product space of the form
$\Omega\times X$. In the scalar case, $X$ is the space of continuous functions on an
interval $[-r,0]$, where $r$ is the delay in the equation and, in the
$m$-dimensional case, $X$ is the product of $m$ such spaces. Then, dynamical
techniques are applied in order to prove the existence of a global attractor in
the standard positive cone $\Omega\times X_+$,
which 
is the appropriate set from the biological point of view. This is a basic result
which guarantees the existence of a subset approached by the trajectories in
$\Omega\times X_+$ as time evolves. When persistence is assumed, then there also exists a
global attractor in the interior of the positive cone $\Omega\times \Int X_+$.

Describing the structure of an attractor is a difficult task in general. Here,
we focus on finding new conditions which imply that the latter attractor is as
simple as it can be. If one thinks of an autonomous equation, that means a
globally attracting equilibrium point. The counterpart of this simple situation
in the non-autonomous setting is a unique attracting invariant set
$K\subset \Omega\times \Int X_+$ which is a copy of the base flow $\W$, that is,
$K=\{(\w,b(\w))\mid\w\in \Omega\}$ for a continuous map $b:\Omega\to \Int X_+$. For the initial
Nicholson model, this implies the existence of a unique attracting positive
almost periodic solution, so that the behaviour of any other positive solution
is asymptotically almost periodic.

It is important to mention that, although the Nicholson systems are not
cooperative, they induce a local monotone and concave skew-product semiflow
$\tau$ in a neighbourhood of $\Omega\times \{0\}$.  This allows us to apply standard methods
of comparison of solutions for the usual order in $X_+$, as well as the general
theory for monotone and concave skew-product semiflows by N\'{u}\~{n}ez et
al.~\cite{nuos4}.

The conditions in this work guaranteeing the existence of a unique attracting
positive almost periodic solution are complementary to those in the
literature. This fact extends the scope of applicability of our results and
allows us to illustrate the structure of the global attractor with the aid of
numerical techniques in these new cases. Moreover, the behaviour of the global
attractors when the migration and mortality rates of the Nicholson system vary
is investigated numerically.

We briefly describe the structure of the paper. In Section~\ref{sec-preli}, we
include some preliminaries to make the paper reasonably self-contained. In
Section~\ref{sec-nicholson}, the theoretical results on the existence of attractors
for Nicholson systems are presented and, under the hypothesis of uniform
persistence, precise inequalities are given which imply that the attractor in
the interior of the positive cone is an exponentially stable copy of the
base. Results of this type have also been obtained by
Faria~\cite{faria17,faria21}. Finally, in Section~\ref{sec-numerical}, some numerical
simulations illustrate the applicability of our results and compare them with
other results given in the previous literature.
\section{Some preliminaries}\label{sec-preli}
We include some basic concepts of the theory of non-autonomous dynamical systems
relevant in this work. In Section~\ref{sec-nicholson}, we will provide a
detailed explanation of the process to address the study of Nicholson systems
in this context, by means of the hull construction.

Let $(\Omega,d)$ be a compact metric space. A real {\em continuous flow\/}
$(\Omega,\sigma,\R)$ is defined by a continuous map
$\sigma: \R\times \Omega \to \Omega,\; (t,\w)\mapsto \sigma(t,\w)=\sigma_t(\w)=\w{\cdot}t$ satisfying (i)
$\sigma_0=\text{Id}$, and (ii) $\sigma_{t+s}=\sigma_t\circ\sigma_s$ for all $s$,
$t\in\R$\,.  The set $\{ \w{\cdot}t\mid t\in\R\}$ is called the {\em orbit\/} of the point
$\w$. We say that a subset $\Omega_1\subset \Omega$ is {\em
  $\sigma$-invariant\/} if $\sigma_t(\Omega_1)=\Omega_1$ for every $t\in\R$.  The flow
$(\Omega,\sigma,\R)$ is called {\em minimal\/} if it does not contain properly any other
compact $\sigma$-invariant set, or equivalently, if every orbit is dense.

Given a continuous flow $(\W,\sigma,\R)$ on a compact metric space $\W$ and a
complete metric space $(X,\dd)$, a continuous {\em skew-product semiflow\/}
$(\Omega\times X,\tau,\,\R_+)$ on the product space $\Omega\times X$ is determined by a continuous map
\begin{equation*}
  \begin{array}{cccl}
    \tau \colon  &\R_+\times \Omega\times X& \longrightarrow & \Omega\times X \\
              & (t,\w,x) & \mapsto &(\w{\cdot}t,u(t,\w,x))
  \end{array}
\end{equation*}
which preserves the flow on $\Omega$, called the {\em base flow\/}.  The semiflow
property means (i) $\tau_0=\text{Id}$, and (ii) $\tau_{t+s}=\tau_t \circ \tau_s$ for
$t,s\geq 0,$ where $\tau_t(\w,x)=\tau(t,\w,x)$ for each
$(\w,x) \in \Omega\times X$ and $t\in \R_+$.  This leads to the so-called (nonlinear)
semicocycle property,
\begin{equation}\label{semicocycle}
  u(t+s,\w,x)=u(t,\w{\cdot}s,u(s,\w,x))\quad\mbox{for $t,s\ge 0$ and $(\w,x)\in \Omega\times X$}.
\end{equation}

The set $\{ \tau(t,\w,x)\mid t\geq 0\}$ is the {\em semiorbit\/} of the point
$(\w,x)$. A subset $K$ of $\Omega\times X$ is {\em positively invariant\/} if
$\tau_t(K)\subset K$ for all $t\geq 0$ and it is $\tau$-{\em invariant\/} if
$\tau_t(K)= K$ for all $t\geq 0$.  A compact $\tau$-invariant set $K$ for the semiflow is
{\em minimal\/} if it does not contain any nonempty compact $\tau$-invariant set
other than itself.

Some relevant references for global attractors and pullback attractors in non-autonomous
dynamical systems are Carvalho et al.~\cite{book:CLR} and Kloeden and Rasmussen
\cite{klra}. 
For a skew-product semiflow over a compact base flow $\W$, the {\em global
  attractor\/} $\A\subset \W\times X$, when it exists, is an invariant compact set
attracting a certain class $\mathcal{D}(X)$ of subsets of $\W\times X$ forwards in time;
namely,
\[
  \lim_{t\to\infty} {\rm dist}(\tau_t(\W\times X_1),\A)=0 \quad\text{for each}\; X_1\in \mathcal{D}(X)\,,
\]
for the Hausdorff semidistance. The standard choices for $\mathcal{D}(X)$ are either the
class $\mathcal{D}_b(X)$ of bounded subsets of $X$ or the class
$\mathcal{D}_c(X)$ of compact subsets of $X$ (see Cheban et al.~\cite{chks}).

Since $\W$ is compact, the non-autonomous set $\{A(\w)\}_{\w\in \W}$, formed by the
$\w$-sections of $\A$ defined by $A(\w)=\{x\in X\mid (\w,x)\in \A\}$ for each
$\w\in \W$, is a {\em pullback attractor\/}, that is, $\{A(\w)\}_{\w\in \W}$ is
compact, invariant, and it pullback attracts all the sets $X_1 \in \mathcal{D}(X)$:
\begin{equation}\label{eq:pullback}
  \lim_{t\to\infty} {\rm dist}(u(t,\w{\cdot}(-t),X_1),A(\w))=0\quad \text{for all}\; \w\in \W\,.
\end{equation}

Note that the notion of pullback attractor is well-defined within the context of
semiflows because the time variable in~\eqref{eq:pullback} is positive.


\section{Attractors in Nicholson systems}\label{sec-nicholson}
In this section, we apply the theory of non-autonomous dynamical systems to get
some useful information about the so-called Nicholson systems, regarding the
existence of attractors. We also determine new conditions to ensure a simple
structure of the global attractor, meaning the existence of a unique almost
periodic attracting solution of the Nicholson system. For completeness, recall
that a continuous function $h:\R\to\R$ is almost periodic if, for every
$\varepsilon>0$, the set of the so-called $\varepsilon$-{\em periods\/} of $h$,
$\{s\in\R\mid |h(t+s)-h(t)|<\varepsilon \;\text{for all}\; t\in\R\}$, is relatively dense.

Since the model has been explained in detail in many publications
(e.g.,~see~\cite{faria17}, \cite{faria21},~\cite{nosv}, and~\cite{obsa16}), we do not include
here the details of its history or the biological meaning of the imposed
conditions. We consider an $m$-dimensional system of delay FDEs with patch
structure ($m$ patches) and a nonlinear term of Nicholson type, where the
environment exhibits an almost periodic time variation:
\begin{equation}\label{nicholson delay}
  y_i'(t)=-\wit d_i(t)\,y_i(t) +\des \sum_{j=1}^m \wit a_{ij}(t)\,y_j(t) + \wit\beta_{i}(t)\,y_i(t-r_i)\,e^{-\wit c_i(t)\,y_i(t-r_i)}\,,\quad t\geq 0\,
\end{equation}
for $i=1,\ldots,m$. Here, $y_i(t)$ denotes the density of the population on patch
$i$ at time $t\geq 0$ and $r_i>0$ is the maturation time on that patch. The
coefficient $\wit a_{ij}(t)$ stands for the migration rate of the population
moving from patch $j$ to patch $i$ at time $t\geq 0$. Finally, the nonlinear term
is the delay Nicholson term. We make the following assumptions on the
coefficient functions:
\par\smallskip
\begin{enumerate}[label=\upshape(\text{a$\arabic*$}),series=nicholson_properties,leftmargin=27pt]
\item\label{a1} $\;\wit d_i(t)$, $\wit a_{ij}(t)$, $\wit c_i(t)$ and
  $\wit \beta_{i}(t)$ are almost periodic functions on $\R$;
\item\label{a2} $\;\wit d_i(t)\geq d_0>0$ for each $t\in\R$ and $i\in\{1,\ldots,m\}$;
\item\label{a3} $\;\wit a_{ij}(t)$ are all nonnegative functions and
  $\wit a_{ii}$ is identically null;
\item\label{a4} $\;\wit\beta_{i}(t)>0$ for each $t\in \R$ and $i\in\{1,\ldots,m\}$;
\item\label{a5} $\;\wit c_i(t)\geq c_0>0$ for each $t\in\R$ and $i\in\{1,\ldots,m\}$;
\item\label{a6} $\;\wit d_i(t)-\sum_{j=1}^m \wit a_{ji}(t)>0$ for each $t\in \R$ and
  $i\in\{1,\ldots,m\}\,$.
\end{enumerate}
\par\smallskip\noindent
Although the procedure to build the {\em hull\/} of the Nicholson system has
recently been explained in detail in~\cite{nosv}, we include it here for the sake of
completeness. Take $X=C([-r_1,0])\times \ldots \times C([-r_m,0])$ with the usual cone of
positive elements, denoted by $X_+$, and the supremum norm. Namely,
$X_+=\{\phi\in X\mid \phi_i(s)\geq 0 \; \text{for}\; s\in [-r_i,0], \;1\leq i\leq m\}$ with interior
$\Int X_+=\{\phi\in X\mid \phi_i(s)> 0 \; \text{for}\; s\in [-r_i,0],\;1\leq i\leq m\}$. Then,
$X$ is a strongly ordered Banach space. Note that, for $y\in \R^m$, $y\geq 0$ means
that all components are nonnegative and $y\gg 0$ means that all components are
positive.  The induced partial order relation on $X$ is then given by:
\begin{equation*}
  \begin{split}
    \phi\le \psi \quad &\Longleftrightarrow \quad \psi-\phi\in X_+\,;\\
  \phi\ll \psi \quad &\Longleftrightarrow \quad \psi-\phi\in \Int X_+\,.\qquad\quad\quad~
\end{split}
\end{equation*}

The usual notation is that, given a continuous map $y:[-r,\infty)\to \R^m$ for
$r:=\max(r_1,\ldots,r_m)$ and a time $t\geq 0$, $y_t$ denotes the map in $X$ defined by
$(y_t)_i(s)=y_i(t+s)$, $s\in [-r_i,0]$, for each component $i=1,\ldots,m$. Let us
write~\eqref{nicholson delay} as $y_i'(t)=f_i(t,y_t)$, $1\leq i\leq m$, for the maps
$f_i:\R\times X\to \R$,
\begin{equation*}
  f_i(t,\phi)=- \wit d_i(t)\,\phi_i(0) +\des \sum_{j=1}^m  \wit a_{ij}(t)\,\phi_j(0) + \wit \beta_{i}(t)\,\phi_i(-r_i)\,e^{-\wit c_i(t)\,\phi_i(-r_i)}\,.
\end{equation*}
Consider the map $l:\R\to\R^N$ given by all the almost periodic coefficients
$l(t)=(\wit d_i(t),\wit a_{ij}(t),\wit \beta_i(t),\wit c_i(t))$ and let $\W$ be its
hull, that is, the closure of the time-translates of $l$ for the compact-open
topology. Then, $\W$ is a compact metric space thanks to the boundedness and
uniform continuity of almost periodic maps. Besides, the shift map
$\sigma:\R\times \W\to\W$, $(t,\w)\mapsto \w{\cdot}t$, with
$(\w{\cdot}t)(s)= \w(t+s)$, $s\in\R$, defines an almost periodic and minimal flow. By
considering the continuous nonnegative maps
$d_i,\,a_{ij},\,\beta_{i},\,c_i:\W\to\R$ such that
$(d_i(\w), a_{ij}(\w), \beta_i(\w), c_i(\w))=\w(0)$, the initial system is included in
the family of systems over the hull, which can be written for each $\w\in\W$~as
\begin{equation}\label{nicholson delay hull}
  y'_i(t)=- d_i(\w{\cdot}t)\,y_i(t) +\des \sum_{j=1}^m  a_{ij}(\w{\cdot}t)\,y_j(t) + \beta_{i}(\w{\cdot}t)\,y_i(t-r_i)\,e^{-c_i(\w{\cdot}t)\,y_i(t-r_i)}\,
\end{equation}
for $i=1,\ldots,m$.  For each $\w \in \W$ and $\varphi\in X$, the solution of~\eqref{nicholson delay
  hull} with initial value $\varphi$ is denoted by $y(t,\w,\varphi)$. The solutions induce a
skew-product semiflow $\tau:\R_+\times \W\times X\to \W\times X$,
$(t,\w,\varphi)\mapsto (\w{\cdot}t,y_t(\w,\varphi))$ (in principle only locally defined). This semiflow
has a trivial minimal set $\W\times \{0\}$, as the null map is a solution of all the
systems over the hull.

Note that this family of systems does not satisfy the standard {\it
  cooperative\/} or {\it quasimonotone condition\/}, which for a single system
of delay FDEs $y'(t)=g(t,y_t)$ given by a map $g:\R_+\times X\to \R^m$ reads as: whenever
$\phi\leq \psi$ and $\phi_i(0)=\psi_i(0)$ for some $i$, then
$g_i(t,\phi)\leq g_i(t,\psi)$ for all $t\geq 0$.  This condition implies that ordered initial
data $\phi\leq \psi$ lead to ordered solutions, as far as defined (e.g., see
Smith~\cite{smit}). In any case, the set $\W\times X_+$ is invariant for the dynamics,
that is, the solutions of~\eqref{nicholson delay hull} starting inside the positive
cone remain inside the positive cone while defined
(see~\cite[Theorem~5.2.1]{smit}). Also, if $\varphi\geq 0$ with
$\varphi(0)\gg 0$, then $y(t,\w,\varphi)\gg 0$ for all $t\geq 0$. Besides, the induced semiflow is
globally defined on $\W\times X_+$, since all the solutions of~\eqref{nicholson delay hull}
are bounded (see~\cite[Theorem~3.3]{obsa18}).

The advantage when a global attractor exists for the induced semiflow is that
there is a dynamical object approached by the trajectories in the product space
$\W\times X_+$ in a forward sense.  In our first result, we prove that the induced
semiflow always has a global attractor. To the best of our knowledge, this
general result has not been stated before. Recall that $\mathcal{D}_b(X_+)$ stands for the
class of bounded sets in $X_+$. Also, given a constant $\rho>0$, we denote by
$\bar\rho$ either the vector in $\R^m$ or the map in $X$ which takes the constant
value $\rho$ in all components.

\begin{theorem}\label{teor:global atractor Ni}
  Assume that the Nicholson system~\eqref{nicholson delay}
  satisfies~\ref{a1}--\ref{a6}. Then, there is a global attractor
  $\A\subset \W\times X_+$ with respect to the class $\mathcal{D}_b(X_+)$ for the induced skew-product
  semiflow $\tau:\R_+\times \W\times X_+\to \W\times X_+$.
\end{theorem}
\begin{proof}
  To get the existence of a global attractor, it suffices to find an absorbing
  compact set (e.g., see \cite[Theorem~1.36]{klra}). Under the regularity conditions
  satisfied by the Nicholson systems, for $r:=\max(r_1,\ldots,r_m)$, the map
  $y_r:\W\times X\to X$ is compact, meaning that it takes bounded sets into relatively
  compact sets. Then, given a constant $\rho>0$, the set
  \begin{equation}\label{eq:def_H}
    H=\cls \big\{y_r(\w,\varphi) \mid \w\in \W\,,\; 0 \leq \varphi \leq \bar \rho \big\}\subset X_+
  \end{equation}
  is compact. Now we search for the appropriate $\rho>0$ so that $\W\times H$ is
  absorbing, that is, for every bounded subset $X_1\subset X_+$ there exists
  $t_1=t_1(X_1)$ such that $\tau_t(\W\times X_1)\subset \W\times H$ for all $t\geq t_1$.

  For the constants $c_0$ given in condition~\ref{a5} and
  $\beta_i^+:=\sup_{t\in\R}\wit \beta_i(t)$, $1\leq i\leq m$, let us consider the family of ODEs
  given by
  \begin{equation}\label{eq:odes}
    z'_i(t)=- d_i(\w{\cdot}t)\,z_i(t) +\des \sum_{j=1}^m  a_{ij}(\w{\cdot}t)\,z_j(t)+ \frac{\beta_i^+}{e\,c_0}\,,\quad 1\leq i\leq m\,,\; \w\in\W \,.
  \end{equation}
  This family is cooperative by \ref{a3} and it is easy to check that it is a
  majorant family of the Nicholson family~\eqref{nicholson delay hull}, considering
  the extrema of the maps $h_i(y)=y\,e^{- c_i(\w)\,y}$ for $y\geq 0$,
  $1\leq i\leq m$, and condition~\ref{a5}. Let $z(t,\w,z_0)$ denote the solution of the
  previous system of ODEs for $\w\in\W$ with initial value $z_0\in\R^m$.

  It can be deduced from condition~\ref{a6} that the homogeneous part of the family
  of ODEs~\eqref{eq:odes} admits an exponential dichotomy with full stable
  subspace. From this, it follows that the solutions of~\eqref{eq:odes} are
  ultimately bounded, uniformly for $\w\in\W$, that is, there exists $\rho>0$ such
  that, given an initial condition $z_0\in\R^m$, there is a $t_0=t_0(z_0)$ such that
  $z(t,\w,z_0)\leq \bar \rho$ for all $\w\in\W$ and $t\geq t_0$. We refer the reader to the
  proof of~\cite[Theorem~6.1]{obsa16} for all the details.

  To finish, let us see that this value of $\rho$ serves our purposes. Given a
  bounded subset $X_1\subset X_+$, we can take $\varphi_0\in X_+$ which satisfies
  $\varphi\leq \varphi_0$ for all $\varphi\in X_1$. For $z_0=\varphi_0(0)$, we take the corresponding
  $t_0$ as in the previous paragraph. Then, by comparing the solutions
  (see~\cite[Theorem~5.1.1]{smit}),
  $y(t,\w,\varphi)\leq z(t,\w,\varphi(0))\leq z(t,\w,\varphi_0(0))\leq \bar \rho$ for all
  $(\w,\varphi)\in \W\times X_1$ and $t\geq t_0$. That is,
  $0\leq y_t(\w,\varphi)\leq \bar \rho$ for all $t\geq t_0+r$. Now, it suffices to apply the
  semicocycle property~\eqref{semicocycle} to conclude that, if $t\geq t_1:= t_0+2r$, then
  $y_t(\w,\varphi)=y_r(\w{\cdot}(t-r),y_{t-r}(\w,\varphi))\in H$ for all
  $(\w,\varphi)\in\W\times X_1$, as wanted.  We are done with the proof.
\end{proof}
In some situations, we can give a description of the global attractor. Often the
interest is in the extinction {\it versus\/} the persistence of the
species. Regarding the extinction at an exponential rate, Novo et al.~\cite{noos07}
have proved that the uniform exponential stability of the null solution is
equivalent to the uniform exponential stability of the null solution of the
linearised systems along the null solution,
\begin{equation}\label{nicholson delay lineal}
  z_i'(t)=-d_i(\w{\cdot}t)\,z_i(t) +\des \sum_{j=1}^m  a_{ij}(\w{\cdot}t)\,z_j(t) + \beta_{i}(\w{\cdot}t)\,z_i(t-r_i)\,,\quad t\geq 0\,,
\end{equation}
for $i=1,\ldots,m$, for each $\w\in \W$. In~\cite[Proposition~3.4]{noos07} one can find a
series of equivalent conditions for this behaviour. In particular, it is enough
that the null solution of~\eqref{nicholson delay lineal} is uniformly asymptotically
stable.  In this situation, the global attractor in the positive cone is the
trivial set $\A=\W\times \{0\}$.
\par
Hereafter, we focus on situations in which the population persists. First of
all, we give the definition of persistence for the initial system~\eqref{nicholson
  delay}, meaning that, if there are some individuals on every patch at the
initial time $t=0$, the population will surpass a positive lower bound on all
the patches in the long run. We use the terminology introduced in~\cite{obsa18}.
\begin{definition}\label{defi persistence Nicholson}
  The Nicholson system~\eqref{nicholson delay} is {\em uniformly persistent at
    $0$} if there exists $M>0$ such that for every initial map $\varphi\geq 0$ with
  $\varphi(0)\gg 0$ there exists a time $t_0=t_0(\varphi)$ such that
  \[ y_i(t,\varphi)\geq M \quad \text{for all }\;t\geq t_0 \;\text{ and }\; i=1,\ldots,m \,.\]
\end{definition}
As shown in~\cite[Theorem~3.4]{obsa18}, this dynamical property for the system
implies the uniform persistence of the whole family \eqref{nicholson delay hull},
according to the next definition. Because of this, we say that Nicholson systems
are well-behaved, as this implication is not to be expected in general
(see~\cite{obsa18} for more details).
\begin{definition}\label{defi persistence hull}
  The skew-product semiflow induced by the family of systems~\eqref{nicholson delay
    hull} is {\em uniformly persistent} in the interior of the positive cone
  $\Int X_+$ if there is a map $\psi\gg 0$ such that, for every $\w\in \W$ and every
  initial map $\varphi\gg 0$, there exists a time $t_0=t_0(\w,\varphi)$ such that
  $y_t(\w,\varphi)\geq \psi$ for all $t\geq t_0$.
\end{definition}
When the Nicholson system is uniformly persistent at $0$, the induced semiflow
has an attractor in the interior of the positive cone.
\begin{theorem}\label{teor:atractor cono pos Nich}
  Assume that the Nicholson system~\eqref{nicholson delay} satisfies conditions
  \ref{a1}--\ref{a6} and it is uniformly persistent at $0$. Then, there exists a global
  attractor for $\tau$ in $\W\times \Int X_+$ with respect to the class of compact
  subsets of $ \Int X_+$.
\end{theorem}
\begin{proof}
  As in the proof of Theorem~\ref{teor:global atractor Ni}, we search for a compact
  absorbing set in $\W\times \Int X_+$. First of all, let $H$ be the compact set
  defined in~\eqref{eq:def_H}, and consider
  $ H':=H\cap \{\varphi \mid \varphi\geq \psi\}$ for a certain $\psi\gg 0$. Note that
  $H'$ is compact and it lies within the interior of the positive cone. It
  remains to prove that, making the appropriate choice of $\psi$, for each compact
  subset $X_1\subset \Int X_+$, the set $\W\times H'$ absorbs the set
  $\Omega\times X_1$, that is, there exists $t_0=t_0(X_1)$ such that
  $\tau_t(\Omega\times X_1)\subset \W\times H'$ for $t\geq t_0$. We already know by the aforementioned theorem
  that, given such a subset $X_1$, there is a $t_1=t_1(X_1)$ such that
  $\tau_t(\Omega\times X_1)\subset \W\times H$ for $t\geq t_1$. Clearly, the fact that trajectories starting
  in $X_1$ will eventually surpass a certain $\psi$ is related to the uniform
  persistence property.

  Since Nicholson systems are not cooperative, it is necessary to build an
  auxiliary family of cooperative systems, so that some results that
  enable a comparison of solutions can be applied.~(Note that we have already
  used this technique in the proof of Theorem~\ref{teor:global atractor Ni}.) The
  idea is not new, so we refer the reader to the proof of~\cite[Theorem~6.2]{obsa16}
  for all the details. A family of delay systems is built with the following
  properties: it is cooperative, concave, of class $C^1$ with respect to the
  functional variable, it shares the linearised family along the null
  solution~\eqref{nicholson delay lineal}, and it is a minorant family of~\eqref{nicholson
    delay hull} in the long run. We denote by $z(t,\w,\varphi)$ the solutions of this
  family for each $\w\in\Omega$ and $\varphi\in X_+$. Then, this cooperative family inherits the
  property of uniform persistence from the linearised family, and this happens
  uniformly for $\w\in\W$, that is, there exists $\psi\gg 0$ such that given
  $\varphi_0\gg 0$ there is a time $t_2=t_2(\varphi_0)$ such that
  $z_t(\w,\varphi_0)\geq \psi$ for all $t\geq t_2$ and all $\w\in\Omega$.  The uniformity in
  $\w$ can be justified because the uniform persistence forces the auxiliary
  family into the dynamical situation described in Case A1 in Theorem~3.8
  in~\cite{nuos4}.

  At this point, since $X_1\subset \Int X_+$, we can find a $\varphi_0\gg 0$ such that
  $\varphi_0\leq \varphi$ for all $\varphi\in X_1$, and take the corresponding
  $t_2=t_2(\varphi_0)$. Then, for all $\w\in\Omega$ and $\varphi\in X_1$,
  $\psi\leq z_t(\w,\varphi_0)\leq z_t(\w,\varphi)$ for $t\geq t_2$, by monotonicity. Since we can compare
  these solutions with those of the Nicholson systems from one time on
  (see~\cite[Theorem~5.1.1]{smit}), uniformly for $\w\in\Omega$, we find a time
  $t_3\geq t_2$ such that $\psi\leq y_t(\w,\varphi)$ for all $\w\in\Omega$,
  $\varphi\in X_1$ and $t\geq t_3$. By taking $t_0:=\max(t_1,t_3)$, the proof is finished.
\end{proof}

For Nicholson systems, the compartmental structure and the relations among the
different compartments have a strong influence on the property of uniform
persistence. One crucial fact for the nonlinear and noncooperative Nicholson
systems (among other systems, e.g., see the Mackey and Glass model for
hematopoiesis~\cite{magl}) is that its uniform persistence turns out to be
equivalent to the uniform persistence of the linearised systems along the null
solution~\eqref{nicholson delay lineal}. Since these linear equations are
cooperative, the general methods in Novo et al.~\cite{noos13} (see also~\cite{obsa16})
to study the uniform persistence of cooperative recurrent non-autonomous delay
FDEs apply, giving a complete spectral characterisation of this dynamical
property.

The next statement is part of~\cite[Theorem~3.5]{obsa18} and is included here for
completeness and because it will be useful in Section~\ref{sec-numerical}. It
offers a characterisation of the uniform persistence of an almost periodic
Nicholson system~\eqref{nicholson delay} in terms of a few Lyapunov exponents, which
can be numerically calculated.
\begin{theorem}\label{teor Nicholson system}
  Assume that the Nicholson system~\eqref{nicholson delay} satisfies
  conditions~\ref{a1}--\ref{a6}, and assume without loss of generality that the
  constant matrix $\bar A=[a_{ij}^+]$ with entries
  $a_{ij}^+ :=\sup_{t\in\R}\wit a_{ij}(t)$ has a block lower triangular structure
  \begin{equation*}
    \left[\begin{array}{cccc}
            \bar A_{11} & 0  &\ldots & 0 \\
            \bar A_{21} & \bar A_{22} & \ldots& 0 \\
\vdots & \vdots  &\ddots & \vdots \\
\bar A_{k1} & \bar A_{k2} & \ldots& \bar A_{kk}
\end{array}\right]\,
\end{equation*}
with irreducible diagonal blocks $\bar A_{jj}$ of dimension $n_j$ for
$j=1,\ldots,k$ $(n_1+\cdots+n_k=m)$. Arrange the set of delays as
$\{r_1,\ldots,r_m\}=\{r^1_1,\ldots,r^1_{n_1},\ldots,r^k_1,\ldots,r^k_{n_k}\}$ and write $X=X^{(1)}\times\ldots\times X^{(k)}$ for
\begin{equation*}
  X^{(j)}= C([-r^j_{1},0])\times \ldots \times C([-r^j_{n_j},0])\,,\quad j=1,\ldots,k\,.
\end{equation*}
For each $j=1,\ldots,k$, consider the $n_j$-dimensional almost periodic linear delay
system
\begin{equation*}
  z_i'(t)=-\wit d_i(t)\,z_i(t) +\des \sum_{l\in I_j} \wit a_{il}(t)\,z_l(t) + \wit\beta_{i}(t)\,z_i(t-r_{i})\,,\quad t\geq 0\,,
\end{equation*}
for $i\in I_j$, the set of indices corresponding to the rows of the block
$\bar A_{jj}$, and let $z^j(t,\bar 1)$ be the solution with initial map
$\bar 1$, the map with all components identically equal to $1$ in the space
$X^{(j)}$. Then, let $\wit\lambda_j$ be defined as
\begin{equation*}
  \wit\lambda_j=\lim_{t\to \infty} \frac{\log\|z_t^j(\bar 1)\|_\infty}{t}\,.
\end{equation*}
\par
Finally, consider the set of indices $I$ associated to the structure of the
linear part of the system as follows: if $k=1$, i.e., if the matrix $\bar A$ is
irreducible, let $I=\{1\}$; else, let
\begin{equation*}
  I=\{j\in\{1,\ldots,k\} \,\mid\, \bar A_{ji}=0 \text{ for all } i\not= j\},\\
\end{equation*}
that is, $I$ is composed by the indices $j$ such that all off-diagonal blocks in
the row of $\bar A_{jj}$ are null.  Then, the almost periodic Nicholson
system~\eqref{nicholson delay} is uniformly persistent at $0$ if and only if
$\wit\lambda_j>0$ for all $j\in I$.
\end{theorem}
Assuming the uniform persistence of the Nicholson system, we give a new result
on the existence of a unique positive almost periodic solution which attracts
every other positive solution as $t\to\infty$. In these cases, the attractor in the
interior of the positive cone is as simple as it can be, i.e.,~a copy of the
base which reproduces the almost periodic dynamics on the base $\W$. Briefly,
whenever the attractor lies within the region of monotonicity of $\tau$ for the
usual ordering, it is a copy of the base. This is a nontrivial generalisation to
the almost periodic case of the same result in the autonomous case:
see~\cite[Theorem~3.1]{faria11}.  Some other related results
are~\cite[Theorem~4.1]{faria17} for a class of periodic Nicholson systems
and~\cite[Theorem~3.4]{faria21}.

We want to note that in the mentioned related results, and in many others in the
literature, there are conditions which imply the uniform persistence of the
system, given in terms of the spectral bound of an associated matrix in the
autonomous case, or by introducing a positive lower bound in expressions of the
type~\eqref{zona inv}. However, we have chosen to directly assume the fact that the
system is persistent, and whenever a particular system is given, calculate the
Lyapunov exponents and check the persistence using the sufficient, but also
necessary, conditions given in Theorem~\ref{teor Nicholson system}.
\begin{theorem}\label{teor-atractor-cooperativo}
  Assume that the Nicholson system~\eqref{nicholson delay} satisfies conditions
  \ref{a1}--\ref{a6} and it is uniformly persistent at $0$. If, for every $t\in\R$,
  \begin{equation}\label{zona inv}
    0< \frac{\wit \beta_i(t)}{\wit d_i(t)-\sum_{j\not=i}  \wit a_{ij}(t)\,\frac{ c_i^+}{c_j^+}} \leq e^{c_i^-/ c_i^+}\,\quad \text{for each } i=1,\ldots, m\,,
  \end{equation}
  for the positive constants $c_i^-:=\inf_{t\in\R}\wit c_i(t)$ and
  $c_i^+:=\sup_{t\in\R}\wit c_i(t)$, then there exists a unique positive almost
  periodic solution of~\eqref{nicholson delay} which attracts every other positive
  solution at an exponential rate; more precisely, it attracts every other
  solution $y(t,\varphi)$ with initial value $\varphi\geq 0$ such that $\varphi(0)\gg 0$.
\end{theorem}
\begin{proof}
  The proof relies on the application of the general theory for monotone and
  concave skew-product semiflows developed in~\cite{nuos4}. Thus, we consider the
  family of systems over the hull~\eqref{nicholson delay hull} and the induced
  skew-product semiflow $\tau$. From condition~\eqref{zona inv}, it follows that for
  every $\w\in \W$ and every $i=1,\ldots, m$,
  \begin{equation*}
    - d_i(\w)+\des \sum_{j=1}^m  a_{ij}(\w)\,\frac{c_i^+}{c_j^+} + \beta_{i}(\w)\,e^{-c_i^-/c_i^+}\leq 0\,.
  \end{equation*}
  Then, it is easy to check that for the constant map $\bar \varphi$ in $X$ with value
  the vector
  \[
    \left(\frac{1}{c_1^+}, \ldots, \frac{1}{c_m^+} \right)\in \R^m\,,
  \]
  the region $\Omega\times [\bar{0},\bar\varphi]$ is positively invariant: just apply the
  criterion given in~\cite[Remark~5.2.1]{smit} for nonquasimonotone delay FDEs,
  bearing in mind the Nicholson nonlinear term. Actually, it is easy to check
  that the restriction of the semiflow to this positively invariant region is
  monotone, concave, and of class $C^1$ with respect to $\varphi$. Besides, recall that
  the persistence property of the initial system implies the uniform persistence
  of the semiflow in the interior of the positive cone. Then, if we fix
  $\w_0\in\W$ and $0\ll\phi_0\leq \bar\varphi$, the omega-limit set of the pair
  $(\w_0,\phi_0)$ is a strongly positive compact and positively invariant set, which
  thus contains a minimal set $K$ such that $0\ll K$ and $\phi\leq \bar\varphi$ for all
  $(\w,\phi)\in K$. Due to the uniform persistence, \cite[Theorem~3.8]{nuos4} implies
  that $K$ is the only strongly positive minimal set for
  $\tau\mid_{\Omega\times [\bar{0},\bar\varphi]}$, it is a copy of the base, namely,
  $K=\{(\w,b(\w))\mid \w\in\ \W\}$ for a continuous map $b:\Omega\to \Int X_+$, and it
  exponentially attracts every other semiorbit for $\w\in\W$ and
  $0\ll\varphi\leq \bar\varphi$, that is, $\lim_{t\to\infty} \|y_t(\w,\varphi)-b(\w{\cdot}t)\|_\infty=0$ exponentially fast.

  Let us now prove that the semiorbit $\tau(t,\w,\varphi)$ of each $\w\in\W$ and
  $\varphi \gg 0$ is attracted by $K$ too. In order to check it, we introduce a majorant
  family of systems which satisfy the quasimonotone condition, are concave, and
  of class $C^1$ with respect to $\varphi$. More precisely, for each
  $1\leq i\leq m$, we consider the map $h_i:\W\times [0,\infty)\to [0,\infty)$ defined by
  \[
    h_i(\w,y)= \left\{
      \begin{array}{ll}
        y\,e^{-c_i(\w)\,y} &\text{if }\; 0\leq y\leq \frac{1}{c_i(\w)}\,,  \\[0,2cm]
        \frac{1}{c_i(\w)\,e} &\text{if }\; y\geq \frac{1}{c_i(\w)} \,,
      \end{array} \right.
  \]
  together with the family of delayed nonlinear systems given for each
  $\w\in \W$ by
  \begin{equation}\label{comparar}
    y'_i(t)=- d_i(\w{\cdot}t)\,y_i(t) +\des \sum_{j=1}^m a_{ij}(\w{\cdot}t)\,y_j(t) + \beta_{i}(\w{\cdot}t)\,h_i(\w{\cdot}t,y_i(t-r_i))\,,
  \end{equation}
  for $i=1,\ldots,m$, where the coefficients are just those of~\eqref{nicholson delay
    hull}. Let $\wit\tau:\R^+\times\W\times X_+\to \W\times X_+$,
  $(t,\w,\varphi)\mapsto (\w{\cdot}t,z_t(\w,\varphi))$ denote the induced skew-product semiflow, where
  $z(t,\w,\varphi)$ is the solution of system~\eqref{comparar} with initial value
  $\varphi\in X_+$. This semiflow turns out to be monotone, concave, and of class
  $C^1$ in $\varphi$. Besides, $K\gg 0$ is also a minimal set for $\wit\tau$, because the
  systems coincide when restricted to $\Omega\times [\bar{0},\bar\varphi]$. Then, in particular
  $\wit\tau$ is globally defined (see~\cite[Proposition~3.6]{nuos4}).  Also, the fact
  that $K$ attracts all the semiorbits starting below it implies that $K$ is the
  only minimal set for $\wit\tau$, and thus attracts all the solutions in the
  interior of the positive cone (see~\cite[Theorem~3.8]{nuos4}), that is, for each
  $\w\in\W$ and $\varphi \gg 0$, $\|z_t(\w,\varphi)-b(\w{\cdot}t)\|_\infty\to 0$ as
  $t\to \infty$ exponentially fast. In other words, there is a global attractor for
  $\wit\tau$ in the interior of the positive cone given by the set
  $K\subset [\bar{0},\bar\varphi]$. Now, as systems~\eqref{comparar} satisfy the quasimonotone
  condition, we can apply a standard argument of comparison of solutions
  (see~\cite[Theorem~5.1.1]{smit}) to get that
  $0\leq y_i(t,\w,\varphi)\leq z_i(t,\w,\varphi)$ for all $\w\in\W$,
  $\varphi\in X_+$ and $t\geq 0$. Hence, it is easy to deduce that also the attractor in
  $\W\times \Int X_+$ for the Nicholson systems is in $[\bar{0},\bar\varphi]$, and thus it
  must be $K$, as desired.

  Finally, when we take $\w_1$ as the element in the hull giving the initial
  system~\eqref{nicholson delay}, we get the positive almost periodic solution
  $b(\w_1{\cdot}t)$ attracting every other solution $y_t(\w_1,\varphi)$ with
  $\varphi\gg 0$. Moreover, if $\varphi\geq 0$ with $\varphi(0)\gg 0$, then
  $y(t,\w_1,\varphi)\gg 0$ for all $t\geq 0$, so that we just need to move forwards in time
  and apply the semicocycle property to obtain the attraction result for these
  initial data. The proof is finished.
\end{proof}
\begin{remark}
  Some other systems in the literature with a similar structure can also be
  treated in the same fashion.  For instance, similar results can be stated for
  useful almost periodic population models which are written as
  \begin{equation*}
    y'_i(t)=- \wit d_i(t)\,y_i(t) +\des \sum_{j=1}^m \wit a_{ij}(t)\,y_j(t) + \wit\beta_{i}(t)\,h_i(t,y_i(t-r_{i}))\,,
  \end{equation*}
  for $i=1,\ldots,m$, with assumption~\ref{a6} on the linear part of the systems, and
  where the nonlinearities are of the form
  \[
    h_i(t,y)=\frac{y}{1+\wit c_i(t)\,y^\alpha}\quad (\alpha \geq 1)\,,\quad t\in \R\,,\,y\in \R_+.
  \]
  See~\cite{obsa18} for more details on the structure of these systems from an
  analytical point of view. For example, the scalar model for the process of
  hematopoiesis for a population of mature circulating cells in~\cite{magl} falls within
  this class.
\end{remark}

\section{Numerical simulations}\label{sec-numerical}

The aim of this section is twofold. First, we will illustrate the results
presented in Section~\ref{sec-nicholson} and compare their applicability with
those in the literature. We will then explore the behaviour of the omega-limit
sets of Nicholson equations when their coefficients undergo certain
variations.\par

Let $\mathbb{T}^2=\left(\mathbb{R}/[0,2\,\pi]\right)^2$ be the two-dimensional torus endowed with the
Kronecker flow $\sigma:\mathbb{R}\times\mathbb{T}^2\to\mathbb{T}^2$,
$(t,\theta_1,\theta_2)\mapsto\sigma_t(\theta_1,\theta_2)=(\theta_1+t,\theta_2+\sqrt{2}\,t)$ (mod
$2\,\pi$). As in the previous sections, we simply write
$\sigma_t(\theta)=\theta{\cdot}t$ for each
$\theta=(\theta_1,\theta_2)\in\mathbb{T}^2$. This flow is minimal because $1$ and
$\sqrt{2}$ are linearly independent over $\Q$.

Let us consider the family of two-dimensional quasi-periodic Nicholson systems
given for each $\theta=(\theta_1,\theta_1)\in \mathbb{T}^2$ by
\begin{equation}\label{eq:numerical}
  \begin{aligned}
    y'_1(t)=&\,- d_1(\theta{\cdot}t)\,y_1(t) + a_{12}(\theta{\cdot}t)\,y_2(t)
    + \beta_{1}(\theta{\cdot}t)\,y_1(t-1)\,e^{-c_1(\theta{\cdot}t)\,y_1(t-1)}\,,\\
    y'_2(t)=&\,- d_2(\theta{\cdot}t)\,y_2(t) + a_{21}(\theta{\cdot}t)\,y_1(t) +
    \beta_{2}(\theta{\cdot}t)\,y_2(t-2)\,e^{-c_2(\theta{\cdot}t)\,y_2(t-2)}\,,
  \end{aligned}
\end{equation}
for $t\geq 0$, determined by the continuous coefficients defined for each $t\in \R$ by
\[
  \begin{aligned}
    c_1(\theta{\cdot}t)=\,& 1, \quad c_2(\theta{\cdot}t)= 0.5 +
    0.2\,p(\theta_1+t)+0.01\,q(\theta_2+\sqrt{2}\,t),\\
    a_{12}(\theta{\cdot}t)=\,& \alpha_{12}\, (0.1 +
    0.03\,p(\theta_1+t)+0.01\,q(\theta_2+\sqrt{2}\,t)), \\
    a_{21}(\theta{\cdot}t)=\,& \alpha_{21}\, (1 +
    0.03\,p(\theta_1+t)+0.01\,q(\theta_2+\sqrt{2}\,t)),\\
    m_1(\theta{\cdot}t)=\,& 1.2,\quad d_1(\theta{\cdot}t)= m_1(\theta{\cdot}t)+a_{21}(\theta{\cdot}t), \\
    m_2(\theta{\cdot}t)=\,& \mu\,(1.9 + 0.02\,p(\theta_1+t)), \quad d_2(\theta{\cdot}t)= m_2(\theta{\cdot}t)+a_{12}(\theta{\cdot}t),\\
    \beta_1(\theta{\cdot}t)=\,& 5 +
    0.03\,p(\theta_1+t)+0.01\,q(\theta_2+\sqrt{2}\,t), \\
    \beta_2(\theta{\cdot}t)=\,& 1 +
    0.03\,p(\theta_1+t)+0.01\,q(\theta_2+\sqrt{2}\,t),
  \end{aligned}
\]
where $\mu$, $\alpha_{12}$, and $\alpha_{21}$ are positive real numbers, and
$p,q\in C(\R,\R)$ are $2\,\pi$-periodic functions (or, equivalently, continuous maps
on the torus $\mathbb{T}$). Recall that quasi-periodic maps are a relevant class within
the set of almost periodic maps.  For each $\theta\in \mathbb{T}^2$, system~\eqref{eq:numerical} may
be seen as a quasi-periodic perturbation of an autonomous Nicholson
system. Notice that $m_1$ and $m_2$ represent the mortality rates.\par

Fix $p(t)=\sin(t)$ and $q(t)=\cos(t)$, $t\in\mathbb{R}$. It is straightforward to check that
conditions~\ref{a1}--\ref{a6} are satisfied for the system for $\theta=(0,0)$, which can be
considered as the initial Nicholson system \eqref{nicholson delay}. Moreover, it can
be checked that condition~\eqref{zona inv} holds as well for $\mu=1$ and
$\alpha_{12}=\alpha_{21}\in\{0.8,1,1.2\}$. The bound given in~\cite[Theorem~3.4]{faria21} fails
for the vector with positive components $(1/c_1^+,1/c_2^+)$ used in the proof of
Theorem~\ref{teor-atractor-cooperativo} (see
Figure~\ref{fig:condition_vs_faria_1_2}). We also performed a parameter sweep over
the grid
\[
  \{0.01\,k\mid k=0,1,\ldots,10000\}\times \{0.01\,k\mid k=0,1,\ldots,10000\}
\]
which seems to indicate that the bound given in~\cite[Theorem~3.4]{faria21} fails
for all the vectors with positive components.
\par

\begin{figure}[h]
  \centering
  \begin{subfigure}[t]{0.45\linewidth}
    \centering \includegraphics[width=0.9\linewidth]{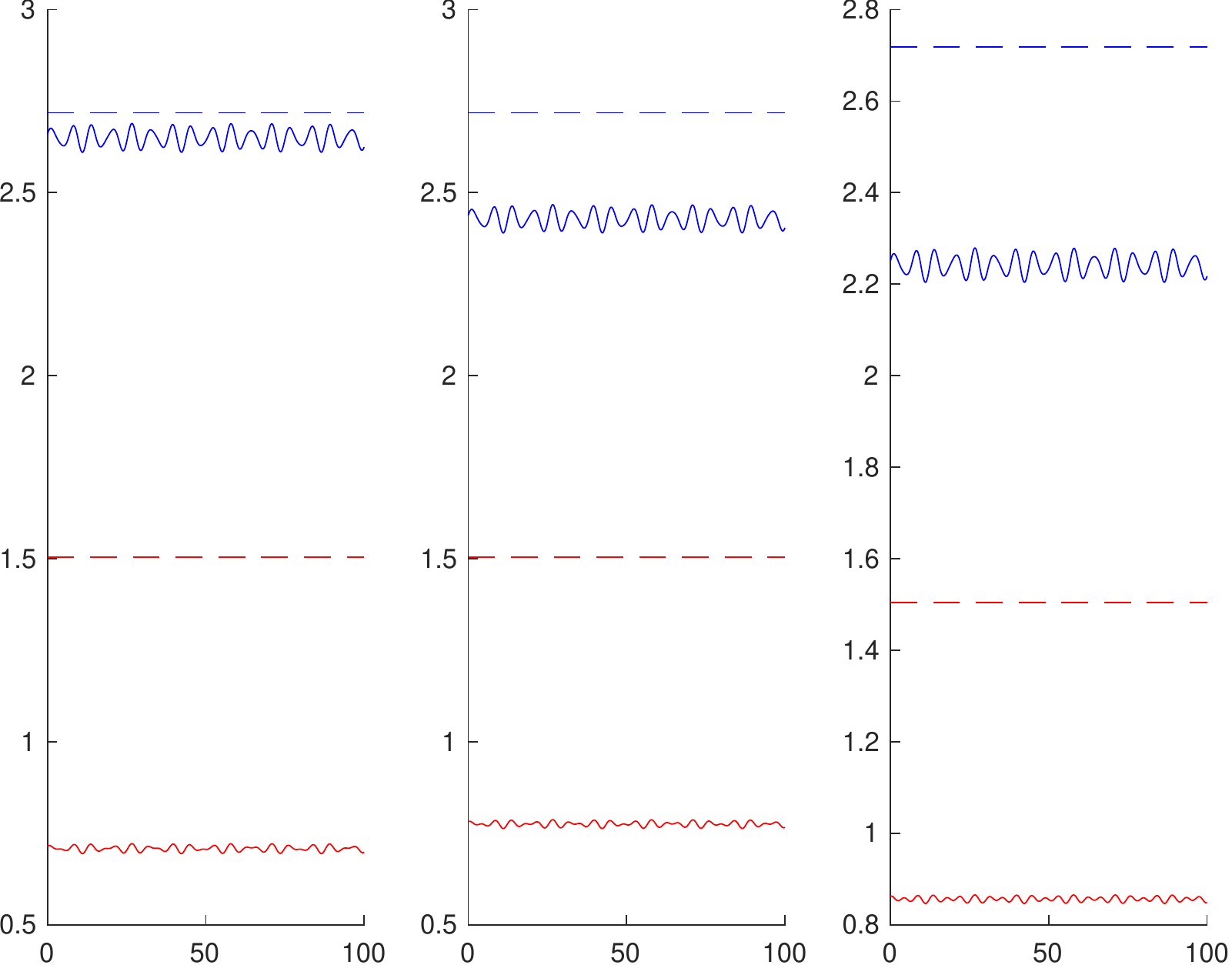}
    \caption{Upper bound (dashed) and middle expression (continuous) of~\eqref{zona inv}.}
  \end{subfigure}
  \hfill
  \begin{subfigure}[t]{0.45\linewidth}
    \centering \includegraphics[width=0.9\linewidth]{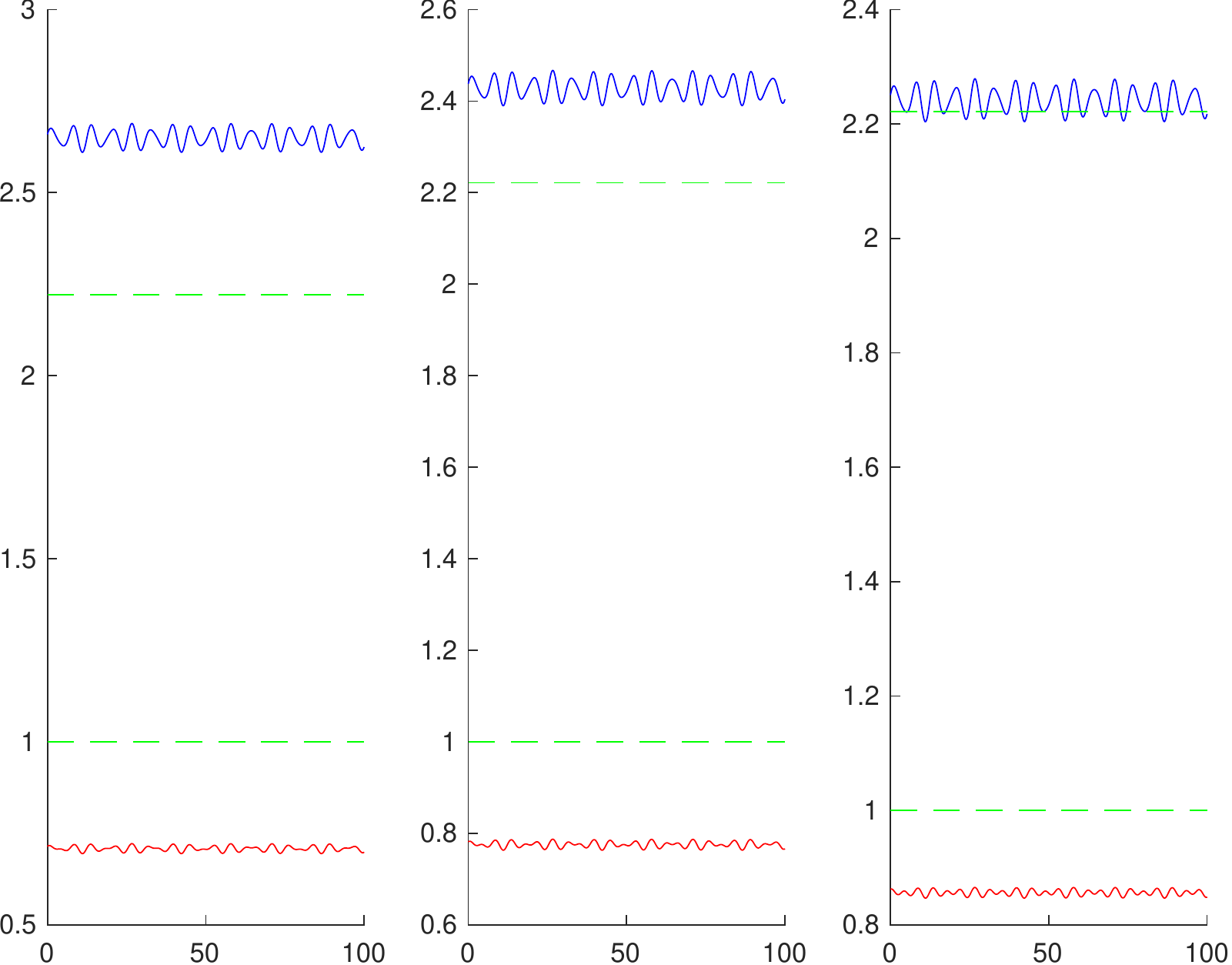}
    \caption{Upper and lower bounds (dashed) and middle expression
      (continuous) of (3.10) in~\cite{faria21} for $v=(1/c_1^+,1/c_2^+)$ .}
  \end{subfigure}
  \caption{
 In both cases,
    $\mu=1$, $p(t)=\sin(t)$, $q(t)=\cos(t)$, and
    $\alpha_{12}=\alpha_{21}=0.8,1,1.2$, resp. The first component is blue and
    the second one is red.}
  \label{fig:condition_vs_faria_1_2}
\end{figure}


In order to apply Theorem~\ref{teor-atractor-cooperativo}, it remains to check that
the quasi-periodic Nicholson system~\eqref{eq:numerical} for $\theta=(0,0)$ is uniformly
persistent at $0$. Assuming the notation of Theorem~\ref{teor Nicholson system}, we
have $\bar A=[\bar A_{11}]$, that is, an irreducible matrix of dimension 2. As a
result, it suffices to check that the Lyapunov exponent $\widetilde \lambda_1>0$.\par

We are in a position to apply an adaptation of the techniques introduced
in~Cal\-za\-da et al.~\cite{caos} to compute $\widetilde \lambda_1$. Specifically,
we will use some appropriate methods to perform the numerical integration of the
delay linear system
\begin{equation}\label{eq:linearized}
  \begin{aligned}
    y'_1(t)=&\,- d_1(\theta{\cdot}t)\,y_1(t) + a_{12}(\theta{\cdot}t)\,y_2(t)
    + \beta_{1}(\theta{\cdot}t)\,y_1(t-1)\,,\\
    y'_2(t)=&\,- d_2(\theta{\cdot}t)\,y_2(t) + a_{21}(\theta{\cdot}t)\,y_1(t) + \beta_{2}(\theta{\cdot}t)\,y_2(t-2)\,
  \end{aligned}
\end{equation}
for $\theta=(0,0)$. Note  that the appropriate state space for this problem is
$X=C([-1,0])\times C([-2,0])$. As suggested by Theorem~\ref{teor Nicholson system}, we
take $\bar 1$ as the initial map, the map in $X$ with both components identically equal to $1$.

A first approach is given by Matlab's code {\itshape dde23} (see Shampine and
Thomp\-son~\cite{shth}), which relies on an explicit Runge-Kutta (2,3) pair of
Bogacki and Sham\-pi\-ne~\cite{bosh}. The results of that
integration present an evident numerical instability, as seen on the left-hand
side of Figure~\ref{fig:instability}. In order to circumvent this issue, the
Gauss-Legendre method of order four for delay equations was considered.  Its
implementation was validated against both the symbolic solution and the
numerical approximation given by Matlab's {\itshape dde23} of the unperturbed
system~\eqref{eq:linearized} with parameters $\mu=\alpha_{12}=\alpha_{21}=1$,
$p=q\equiv 0$. The Gauss-Legendre method is an implicit Runge-Kutta
method with two stages and Butcher tableau
\[
  \begin{array}{c|cc}
    \frac{1}{2} - \frac{1}{6} \sqrt{3} & \frac{1}{4}                  & \frac{1}{4} - \frac{1}{6} \sqrt{3} \\[0.2\normalbaselineskip]
    \frac{1}{2} + \frac{1}{6} \sqrt{3} & \frac{1}{4} + \frac{1}{6} \sqrt{3} & \frac{1}{4} \\[0.2\normalbaselineskip]
    \hline \rule{0pt}{1.2\normalbaselineskip}
                                       & \frac{1}{2}                  & \frac{1}{2}
  \end{array}
\]
\begin{figure}[h]
  \centering \includegraphics[width=0.45\textwidth]{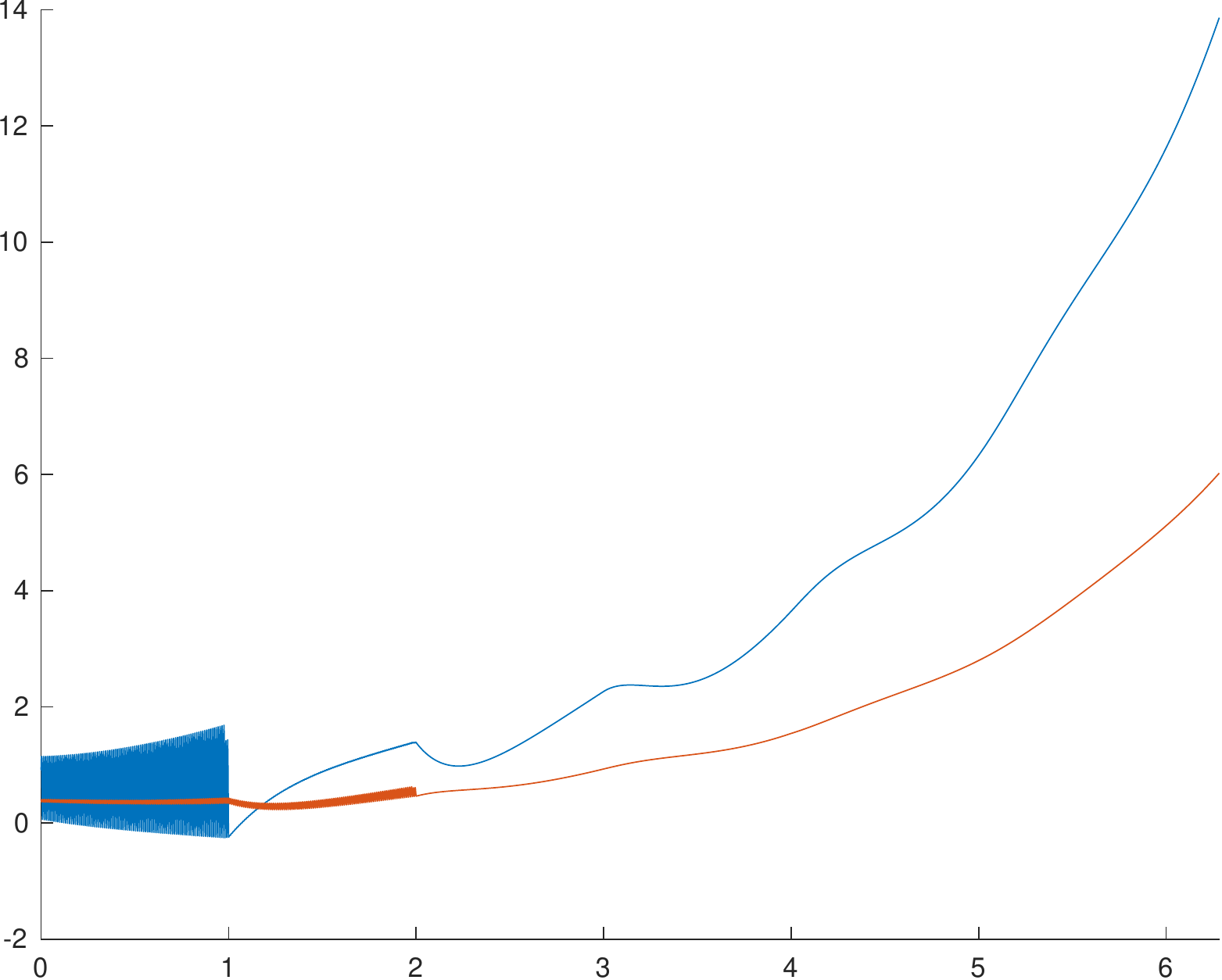} \hfill
  \includegraphics[width=0.45\textwidth]{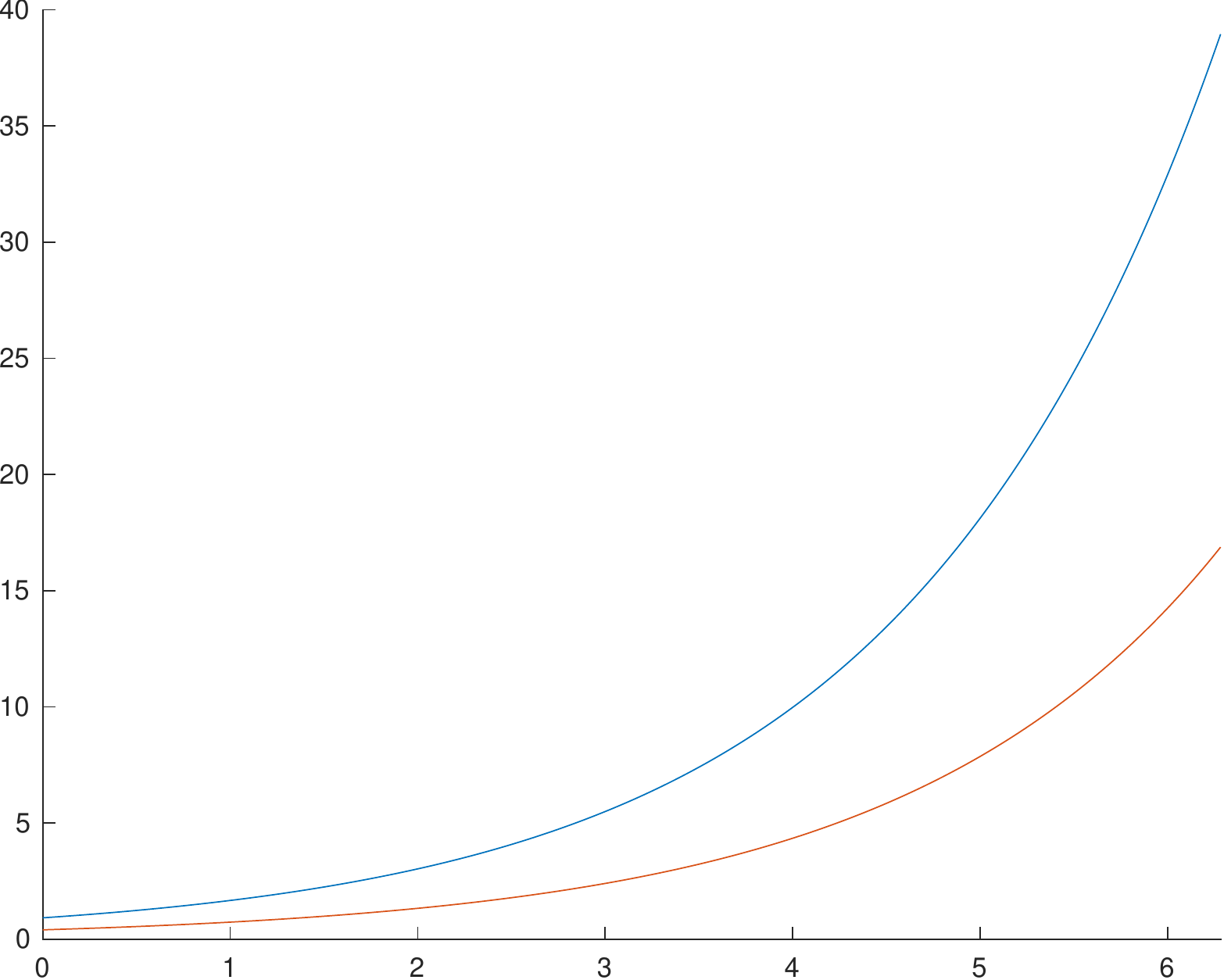}
  \caption{Results of the numerical integration of system~\eqref{eq:linearized}
    on $[-20,2\,\pi]$ with Matlab's {\itshape dde23} (left) and with the
    implicit Gauss-Legendre method with two stages for delay equations (right). The first component is blue and
    the second one is red.}
  \label{fig:instability}
\end{figure}
It is noteworthy that this method has order four and is A-stable as a
consequence of the Wanner-Hairer-N{\o}rsett Theorem (see,
e.g., Iserles~\cite{iser}). The results of the integration of
system~\eqref{eq:linearized} leading to the computation of the required Lyapunov
exponent by the Gauss-Legendre method show no instabilities, as seen on the
right-hand side of Figure~\ref{fig:instability}. Therefore, the techniques
in~\cite{caos} can be applied to conclude that the approximate value of
$\widetilde\lambda_1$ is 0.597, which is positive. Finally, an application of
Theorem~\ref{teor Nicholson system} yields the uniform persistence at $0$ of
system~\eqref{eq:numerical} for $\theta=(0,0)$, as desired. \par

As a consequence, Theorem~\ref{teor:atractor cono pos Nich} implies that the
skew-product semiflow $\tau$ defined by the family~\eqref{eq:numerical},
$\theta\in \mathbb{T}^2$ has a global attractor $K$ in
$\mathbb{T}^2\times \Int X_+$.  Furthermore, by Theorem~\ref{teor-atractor-cooperativo} (see also
its proof), the global attractor $K$ is a copy of the base, that is, there
exists a continuous map $b:\mathbb{T}^2\to\Int X_+$ such that
$K=\{(\theta_1,\theta_2,b(\theta_1,\theta_2))\mid (\theta_1,\theta_2)\in\mathbb{T}^2\}$. The implication in terms of solutions is
that for each $\theta\in \mathbb{T}^2$ there exists a unique positive quasi-periodic solution
of~\eqref{eq:numerical} which attracts every other positive solution at an
exponential rate. This allows us to compute the global attractor $K$, having in
mind that
$\{K(\theta)\}_{\theta\in \mathbb{T}^2}=\{b(\theta)\}_{\theta\in \mathbb{T}^2}$ is the pullback attractor of the semiflow
(see~\eqref{eq:pullback}).\par

Note that the graphs of the components of the map $\mathbb{T}^2\to\mathbb{R}^2$,
$(\theta_1,\theta_2)\mapsto b(\theta_1,\theta_2)(0)$ determine two copies of
$\mathbb{T}^2$. Fix $(\theta_1,\theta_2)\in\mathbb{T}^2$. Then
$b(\theta_1,\theta_2)(0)=\lim_{t\to
  \infty}y(t,\sigma_{-t}(\theta_1,\theta_2),\bar{1})$ and the limit converges
exponentially fast. As a result, we can divide the 2-torus $\mathbb{T}^2$ into a
uniform grid $\{(\theta_1^i,\theta_2^j)\mid i,j=1,\ldots,16\}$ and fix a
tolerance $10^{-6}$. Therefore, we compute
$y^{ij}=y(T,\sigma_{-T}(\theta_1^i,\theta_2^j),\bar{1})$, for each $i,j=1,\ldots,16$, where $T>0$ is such
that the distance between $y^{ij}$ and
$y(T-10,\sigma_{-(T-10)}(\theta_1^i,\theta_2^j),\bar{1})$ is under the
tolerance. This procedure yields an approximation of both copies of
$\mathbb{T}^2$, as shown in Figure~\ref{fig:omega}.\par
\begin{figure}[h]
  \centering \includegraphics[width=0.45\textwidth]{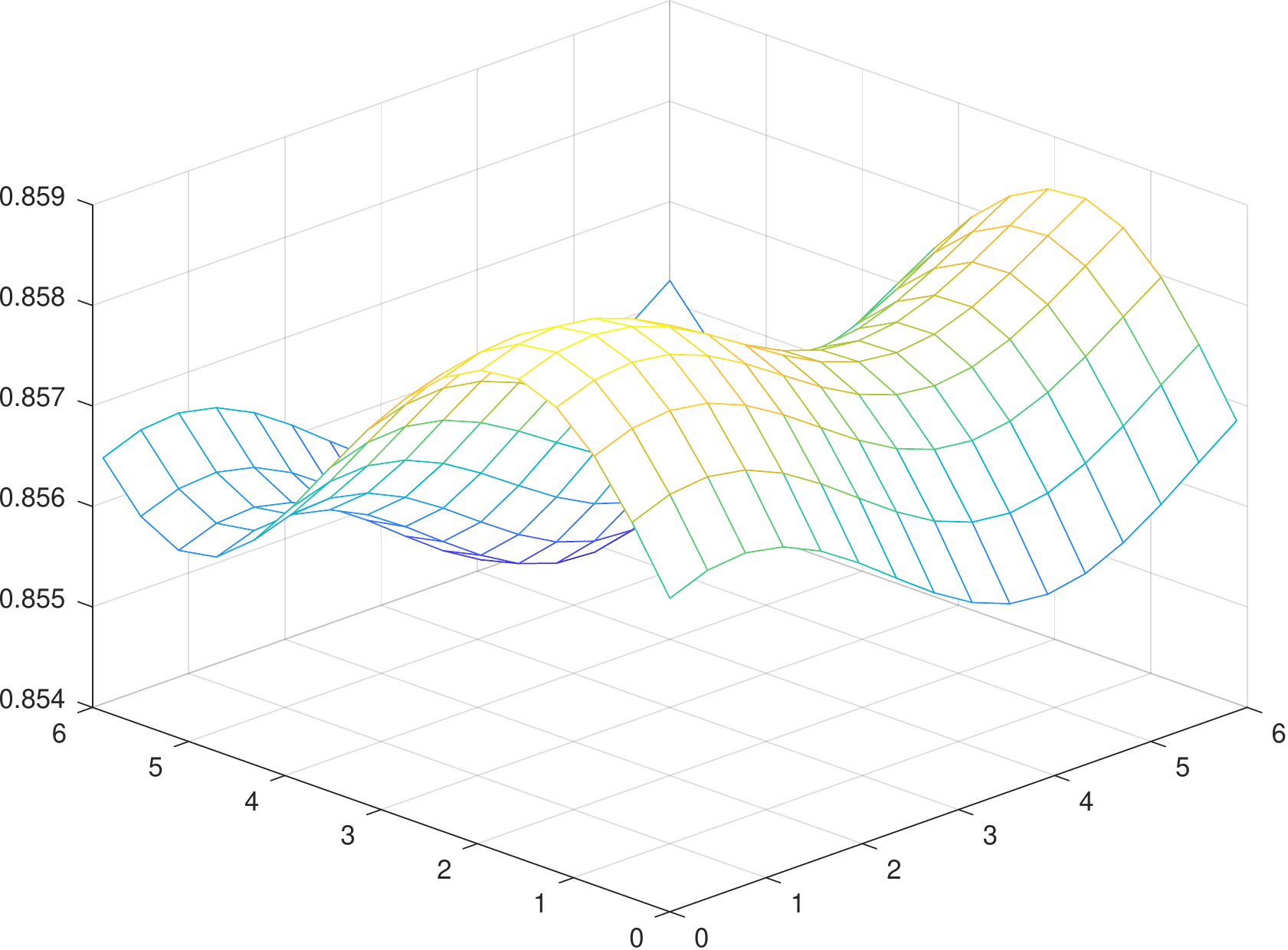} \hfill
  \includegraphics[width=0.45\textwidth]{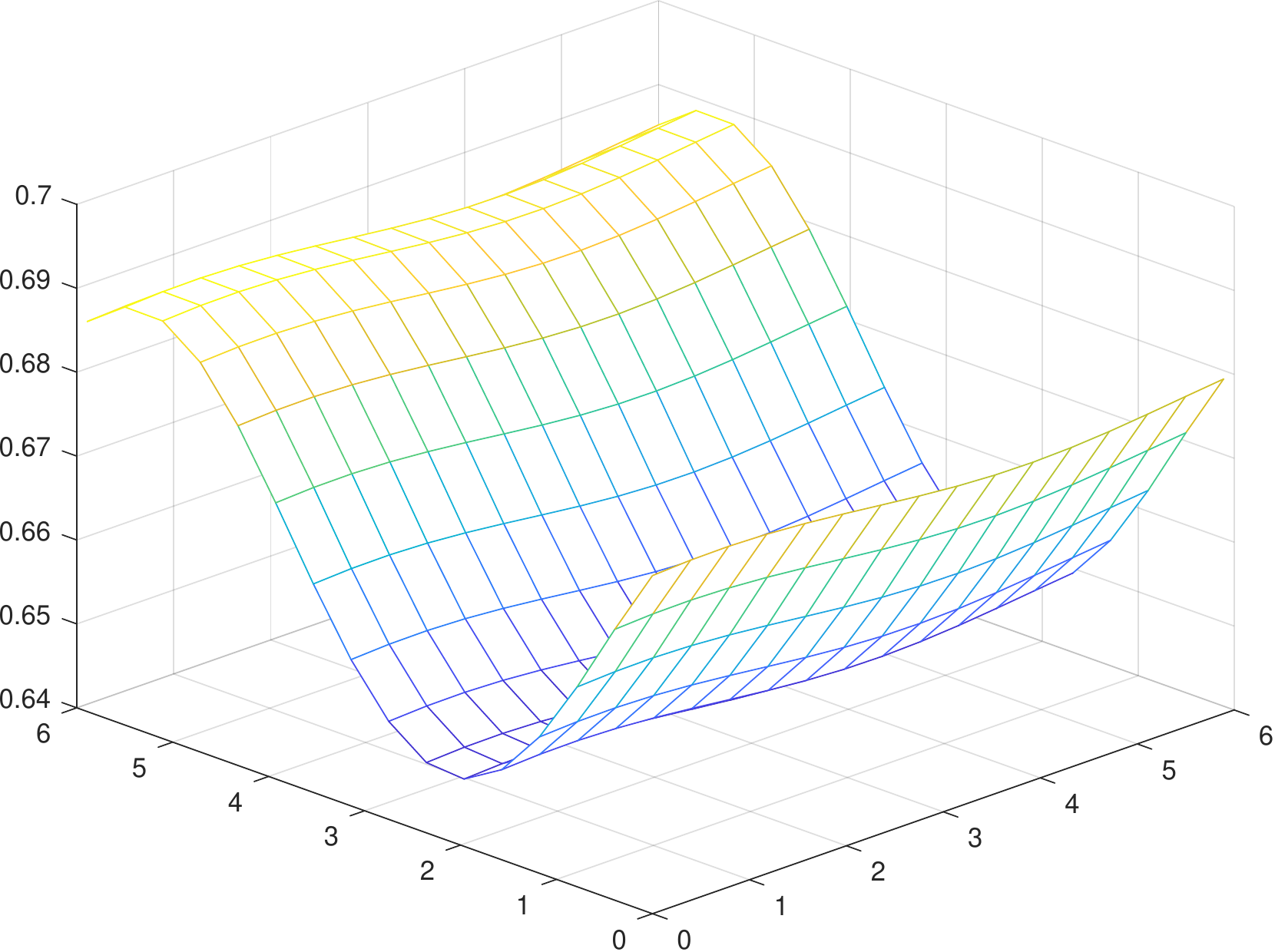}
  \caption{Mesh of points $(\theta_1^i,\theta_2^j,y_1^{ij})$ on the left and
    $(\theta_1^i,\theta_2^j,y_2^{ij})$ on the right, $i,j=1,\ldots,16$, for the
    parameters $\mu=1$, $p(t)=\sin(t)$, $q(t)=\cos(t)$, and
    $\alpha_{12}=\alpha_{21}=1$.}
  \label{fig:omega}
\end{figure}

By repeating the procedure above for the parameters $\mu=1$,
$p(t)=\sin(t)$, $q(t)=\cos(t)$, and
$\alpha_{12}=\alpha_{21}=0.8,1,1.2$, we can see that both components of the
global pullback attractor vary monotonically, either increasingly or
decreasingly, when both migration rates undergo similar variations (see
Figure~\ref{fig:variation_a12_a21}).\par
\begin{figure}[h]
  \centering
  \begin{subfigure}[t]{0.23\linewidth}
    \includegraphics[width=\linewidth]{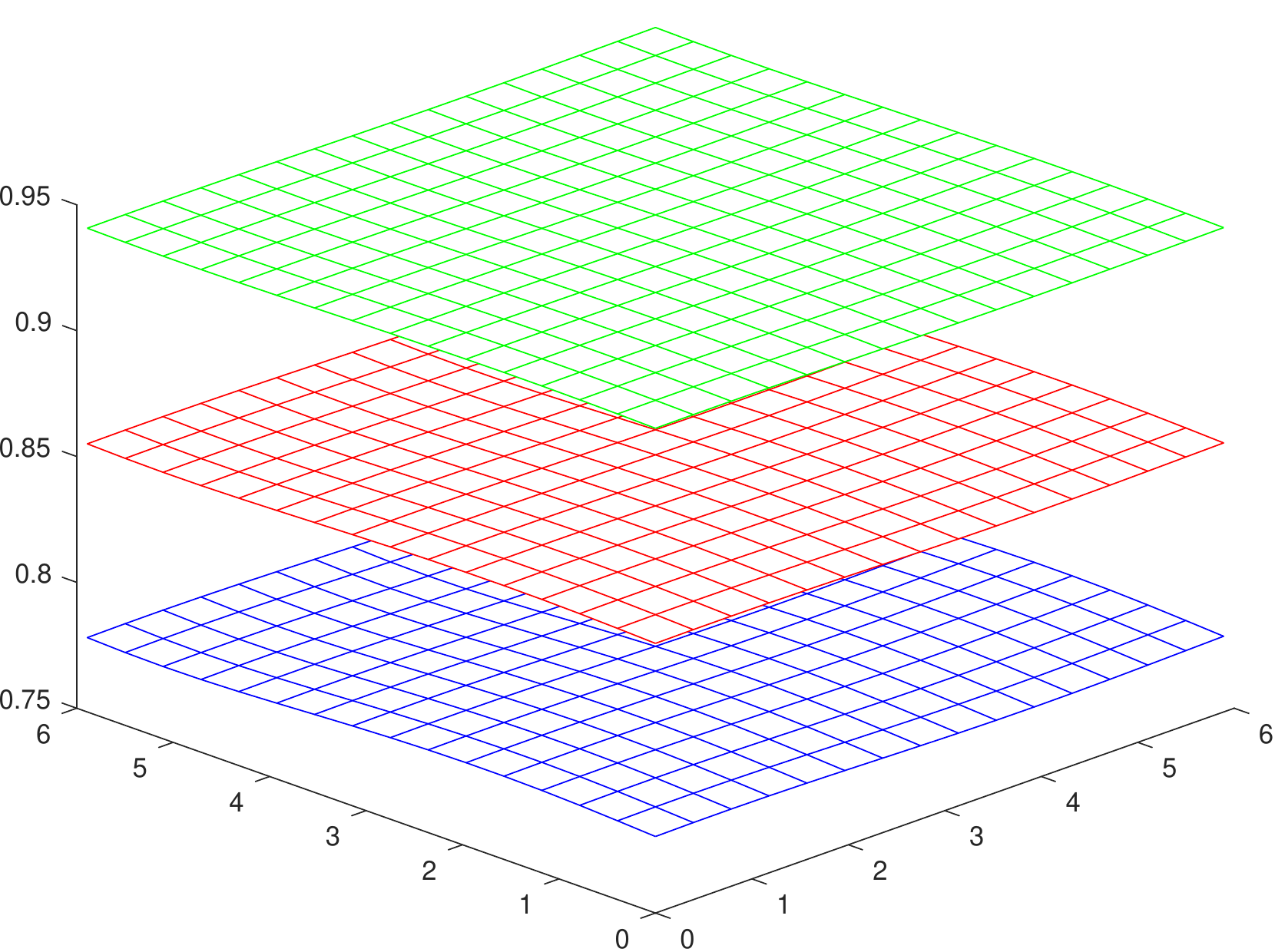}
    \caption{1st component.}
    \label{fig:variation_a12_a21_comp1}
  \end{subfigure}
  \hfill
  \begin{subfigure}[t]{0.23\linewidth}
    \includegraphics[width=\linewidth]{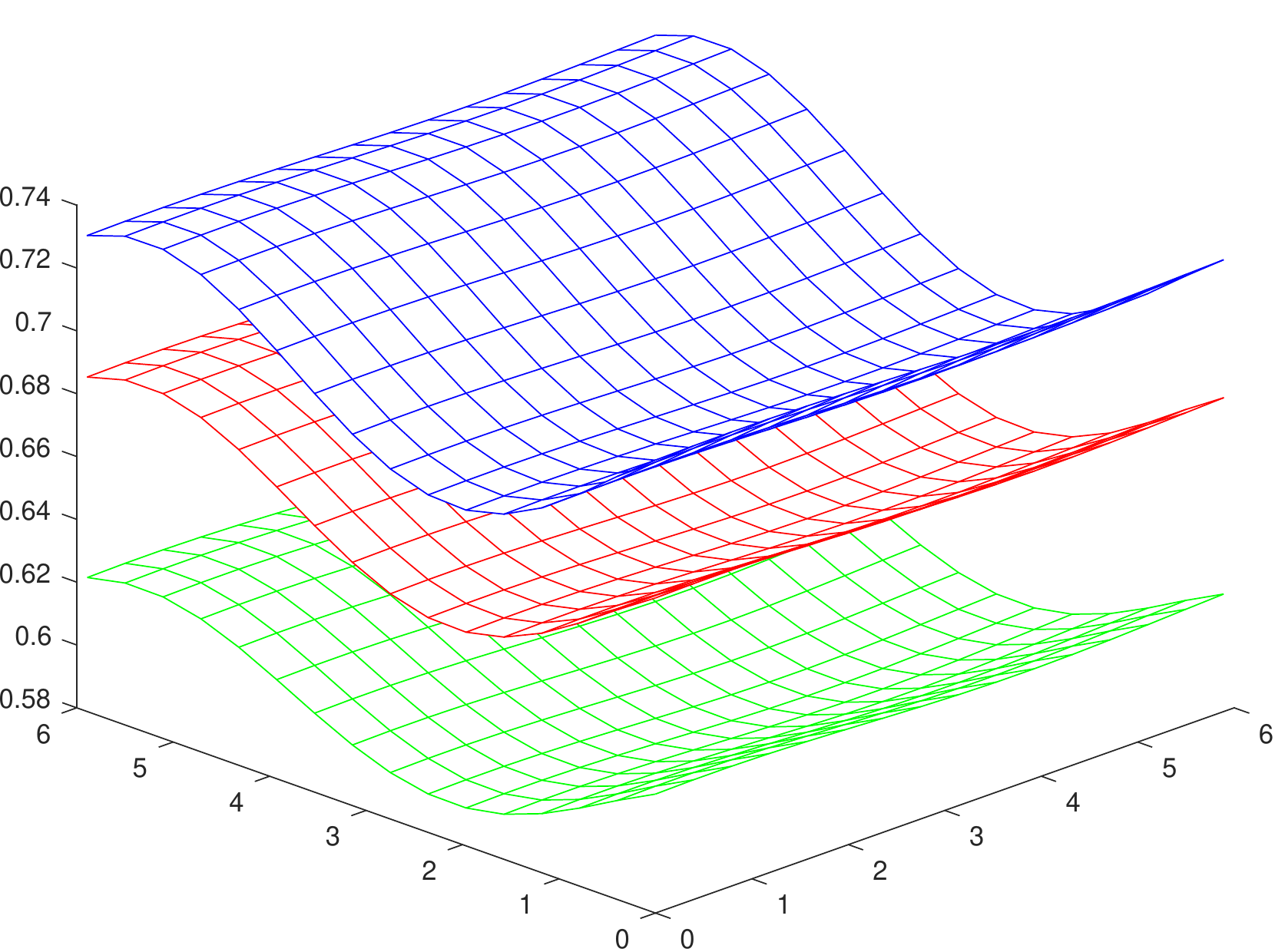}
    \caption{2nd component.}
    \label{fig:variation_a12_a21_comp2}
  \end{subfigure}
  \hfill
  \begin{subfigure}[t]{0.23\linewidth}
    \includegraphics[width=\linewidth]{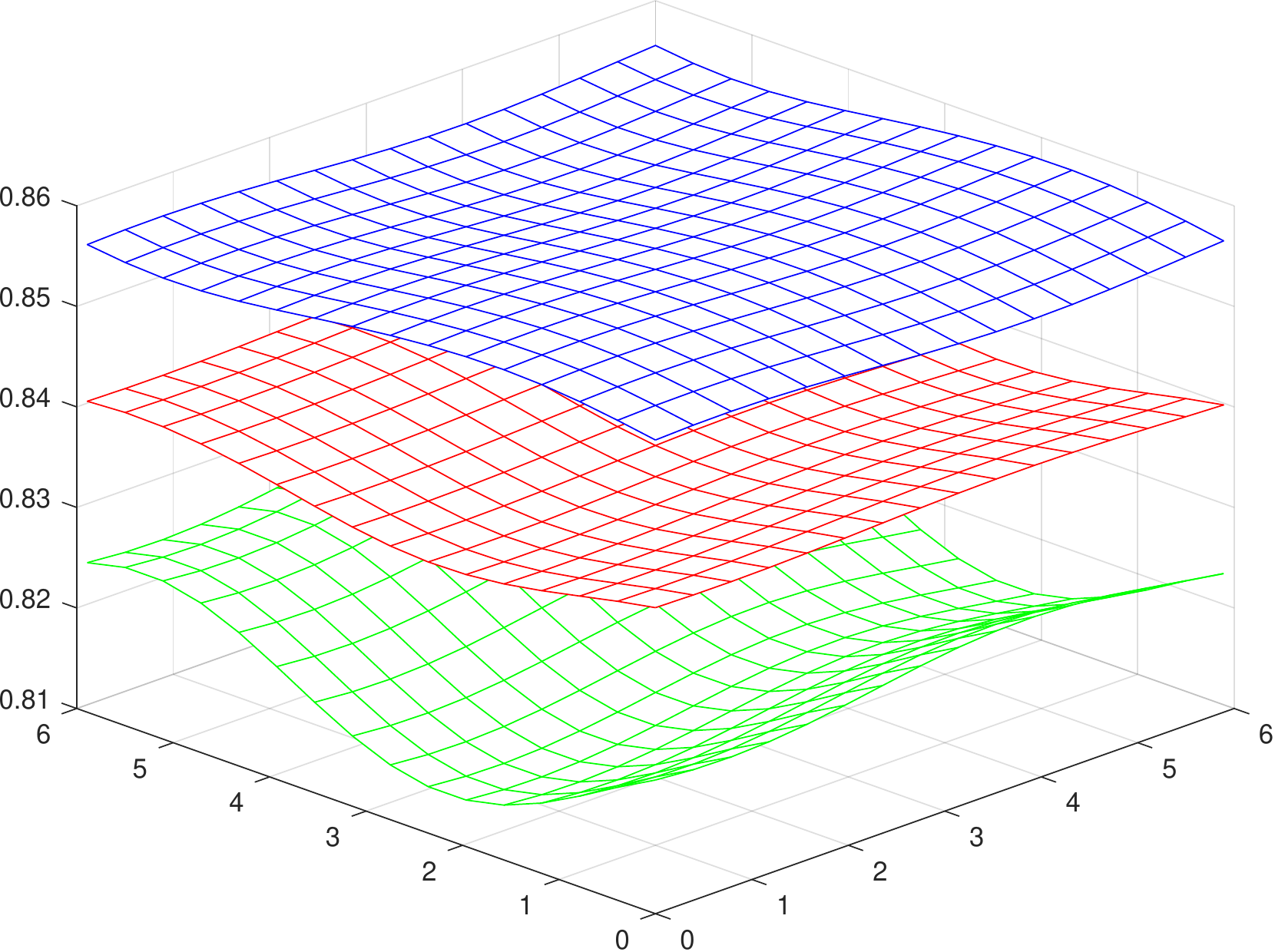}
    \caption{1st component.}
    \label{fig:variation_a12_comp1}
  \end{subfigure}
  \hfill
  \begin{subfigure}[t]{0.23\linewidth}
    \includegraphics[width=\linewidth]{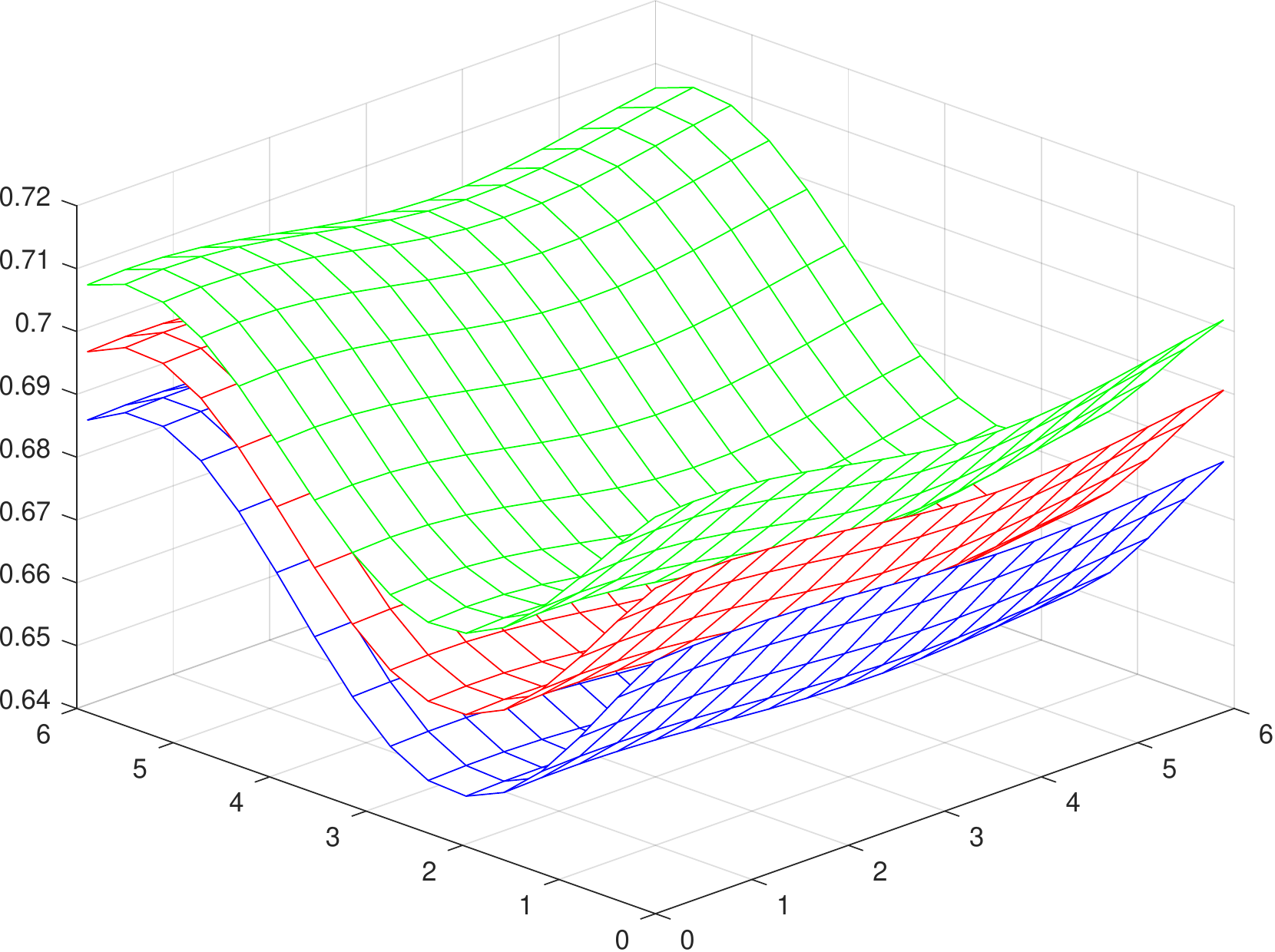}
    \caption{2nd component.}
    \label{fig:variation_a12_comp2}
  \end{subfigure}
  \caption{Variation of the attractor for the parameters $\mu=1$,
    $p(t)=\sin(t)$,
    $q(t)=\cos(t)$. Plots~\ref{fig:variation_a12_a21_comp1}
    and~\ref{fig:variation_a12_a21_comp2}: $\alpha_{12}=\alpha_{21}=0.8$
    (green), $1$ (red), $1.2$ (blue). Plots~\ref{fig:variation_a12_comp1}
    and~\ref{fig:variation_a12_comp2}: $\alpha_{12}=0.01$ (green), $0.5$ (red),
    $1$ (blue).}
  \label{fig:variation_a12_a21}
\end{figure}
If, on the other hand, only one of the migration rates is modified, the
components of the global pullback attractor still vary monotonically, but their
increasing and decreasing characters are reversed (see
Figure~\ref{fig:variation_a12_a21} again). Notice that, in this case, we are
considering the parameters $\mu=1$, $p(t)=\sin(t)$, $q(t)=\cos(t)$,
$\alpha_{21}=1$, and $\alpha_{12}=0.01,0.5,1$, for which it is easy to check
that hypotheses~\ref{a1}--\ref{a6} and condition~\eqref{zona inv} hold.\par

\begin{figure}[h]
  \centering
  \begin{subfigure}[t]{0.23\linewidth}
    \includegraphics[width=\linewidth]{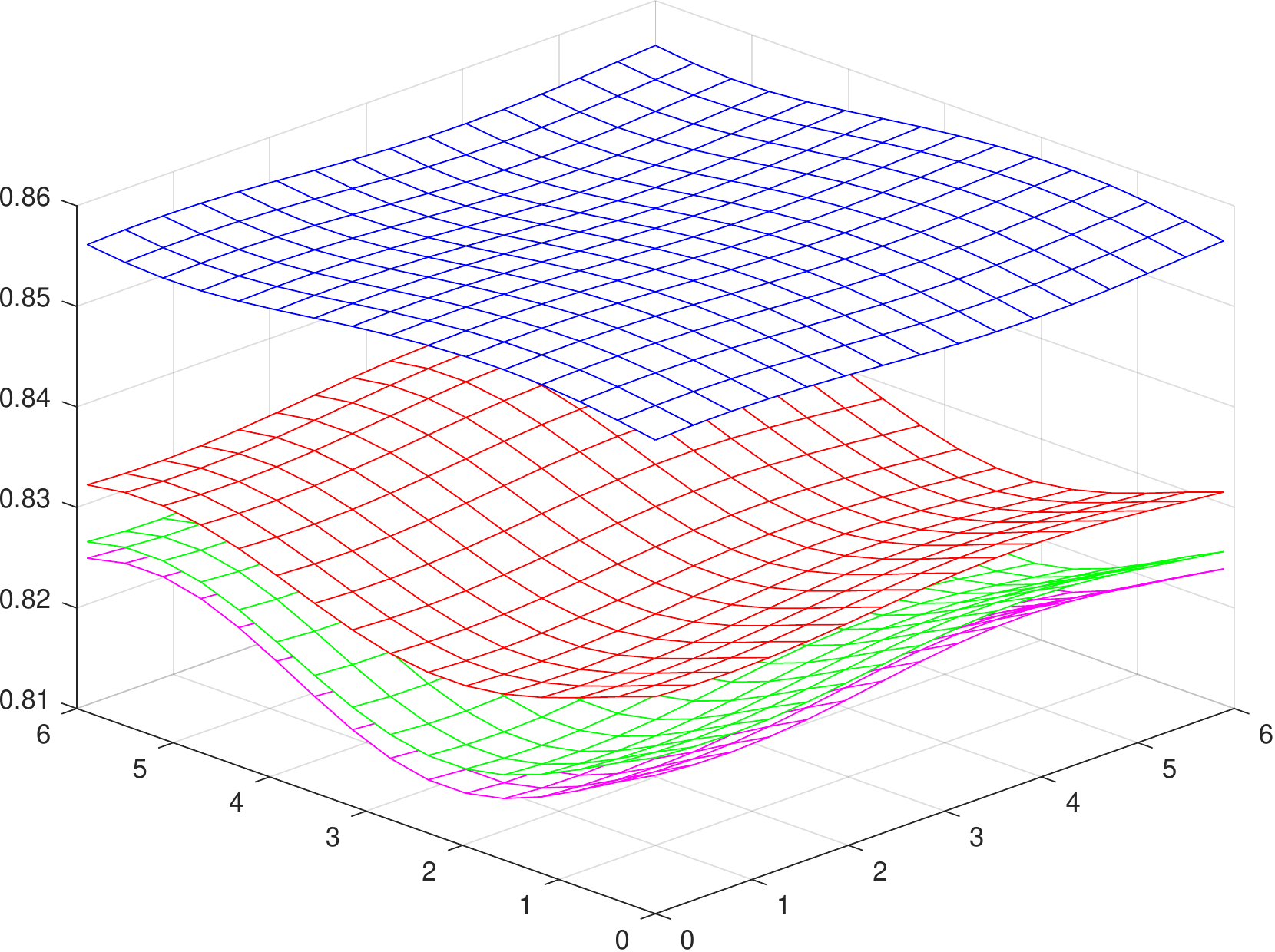}
    \caption{1st component.}
    \label{fig:variation_mortality_increasing_comp1}
  \end{subfigure}
  \hfill
  \begin{subfigure}[t]{0.23\linewidth}
    \includegraphics[width=\linewidth]{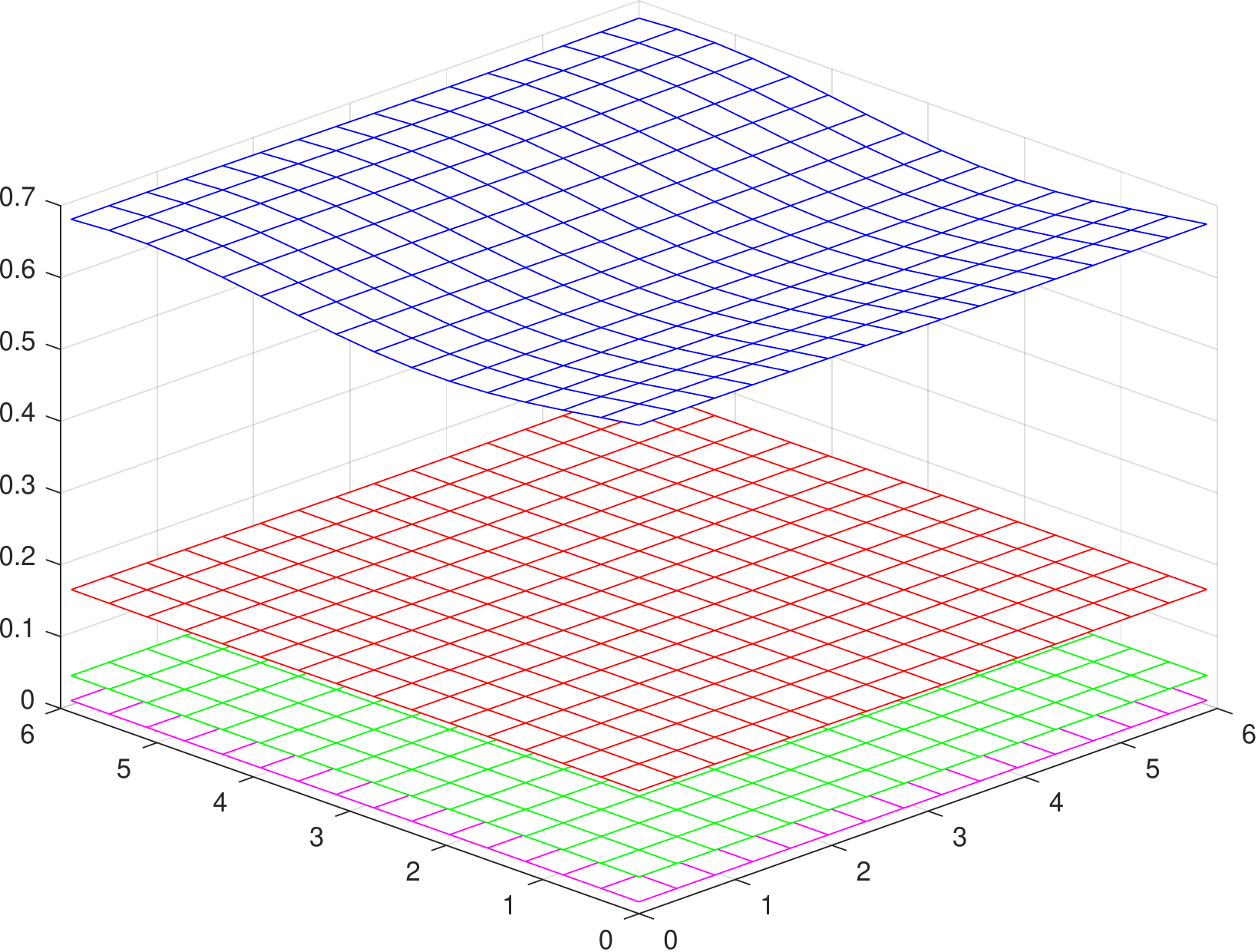}
    \caption{2nd component.}
    \label{fig:variation_mortality_increasing_comp2}
  \end{subfigure}
  \hfill
  \begin{subfigure}[t]{0.23\linewidth}
    \includegraphics[width=\linewidth]{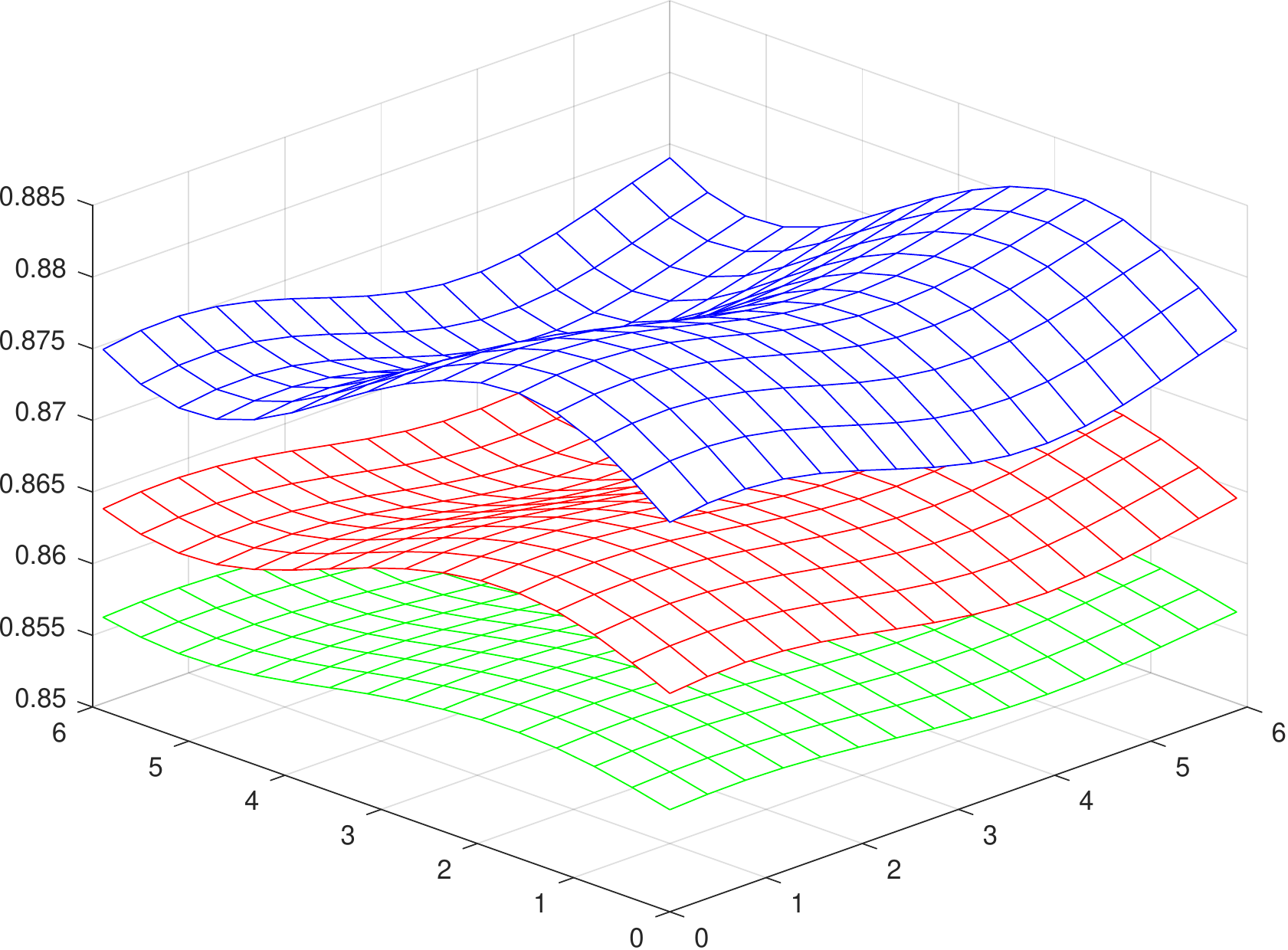}
    \caption{1st component.}
    \label{fig:variation_mortality_decreasing_comp1}
  \end{subfigure}
  \hfill
  \begin{subfigure}[t]{0.23\linewidth}
    \includegraphics[width=\linewidth]{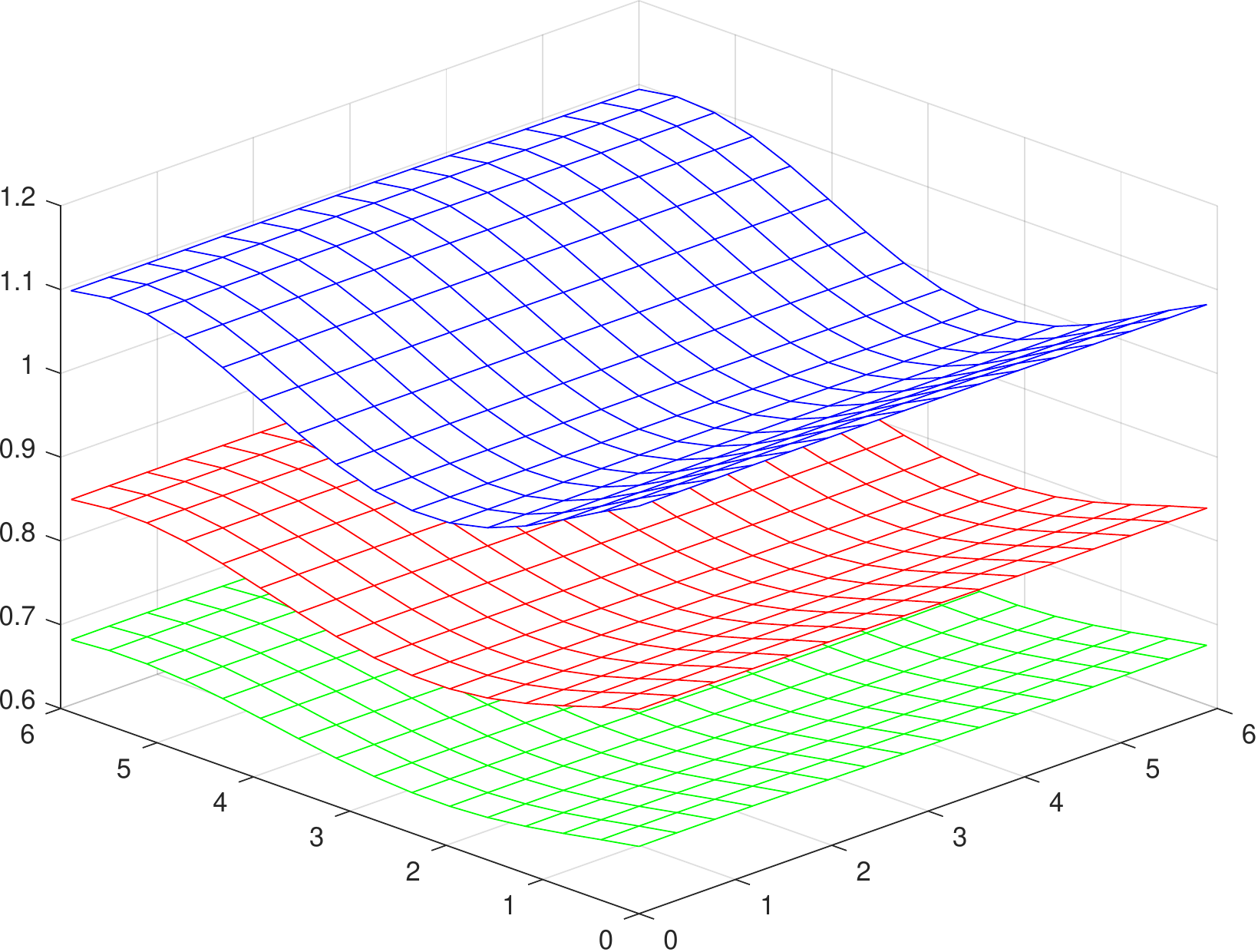}
    \caption{2nd component.}
    \label{fig:variation_mortality_decreasing_comp2}
  \end{subfigure}
  \caption{Variation of the attractor for the parameters
    $p(t)=\sin(t)$, $q(t)=\cos(t)$,
    $\alpha_{21}=\alpha_{12}=1$. Plots~\ref{fig:variation_mortality_increasing_comp1}
    and~\ref{fig:variation_mortality_increasing_comp2}: $\mu=27$ (magenta), $9$
    (green), $3$ (red), $1$
    (blue). Plots~\ref{fig:variation_mortality_decreasing_comp1}
    and~\ref{fig:variation_mortality_decreasing_comp2}: $\mu=1$ (green), $0.85$
    (red), $0.7$ (blue).}
  \label{fig:variation_mortality}
\end{figure}

Finally, let us investigate how the global pullback attractor is modified when
the mortality rate in the second patch is increased or decreased. In either
case, both components of the global pullback attractor vary monotonically
according to the value of the parameter $\mu$ (see
Figure~\ref{fig:variation_mortality}). We are considering the parameters
$\mu=0.7,0.85,1,3,9,27$, $p(t)=\sin(t)$, $q(t)=\cos(t)$,
$\alpha_{21}=\alpha_{12}=1$, for which hypotheses~\ref{a1}--\ref{a6} and
condition~\eqref{zona inv} hold.



\begin{thebibliography}{99}
\bibitem{bosh} {\sc P. Bogacki, L.F. Shampine}, A 3(2) pair of Runge-Kutta
  formulas, {\em Appl. Math. Letters\/}, {\bf 2} (1989), 321--325.

\bibitem{caos} {\sc J.A. Calzada, R. Obaya, A.M. Sanz}, Continuous separation
  for monotone skew-product semiflows: From theoretical to numerical results,
  {\em Discrete Contin. Dyn. Syst. Ser. B} {\bf 20} (3) (2015), 915--944.
\bibitem{book:CLR} \textsc{A. Carvalho, J.A. Langa, J. Robinson}:
  \emph{Attractors for infinite-dimensional non-au\-tono\-mous dynamical
    systems}, \textrm{Springer-Verlag, New York, 2013.}

\bibitem{chks} {\sc D.N. Cheban, P.E. Kloeden, B. Schmalfuss}, The relationship
  between pullback, forwards and global attractors of nonautonomous dynamical
  systems, {\em Nonlinear Dyn. Syst. Theory}, {\bf 2} (2002), 125--144.
\bibitem{faria11} {\sc T. Faria}, Global asymptotic behaviour for a Nicholson
  model with patch struture and multiple delays, {\em Nonlinear Anal.}
  \textbf{74} (2011), 7033--7046.

\bibitem{faria17} {\sc T. Faria}, Periodic solutions for a non-monotone family
  of delayed differential equations with applications to Nicholson systems, {\it
    J. Differential Equations\/} \textbf{263} (2017), 509--533.
\bibitem{faria21} {\sc T. Faria}, Permanence and exponential stability for
  generalised nonautonomous Nicholson systems, {\em Electron. J. Qual. Theory
    Differ. Equ.} \textbf{9} (2021), 1--19.
\bibitem{faos} {\sc T. Faria, R. Obaya, A.M. Sanz}, Asymptotic behaviour for a
  class of non-monotone delay differential systems with applications. {\em
    J. Dyn. Diff. Equat.} \textbf{30} (2018), 911--935.
\bibitem{faro} {\sc T. Faria, G. R\"{o}st}, Persistence, permanence and global
  stability for an n-dimensional Nicholson system, {\em J. Dynamics Differential
    Equations} \textbf{26} (2014), 723--744.
\bibitem{gubl} {\sc W.S.C. Gurney, S.P. Blythe, R.M. Nisbet}, Nicholson's
  blowflies revisited, {\em Nature} \textbf{287} (1980), 17--21.

\bibitem{have} {\sc J.K. Hale, S.M. Verduyn Lunel}, \textit{Introduction to
    Functional Differential Equations}, Applied Mathematical Sciences {\bf 99},
  Springer-Verlag, Berlin, Heidelberg, New York, 1993.

\bibitem{iser} {\sc A. Iserles}, {\em A First Course in the Numerical Analysis
    of Differential Equations\/}, Cambridge University Press, 1996.


\bibitem{klra} {\sc P.E. Kloeden, M. Rasmussen}, {\em Nonautonomous Dynamical
    Systems\/}, AMS Mathematical Surveys and Monographs, Vol. {\bf 176}, AMS,
  Providence, 2011.
\bibitem{magl} {\sc M.C. Mackey, L. Glass}, Oscillation and chaos in
  physiological control systems, {\em Science} \textbf{197} (4300) (1977),
  287--289.
\bibitem{noos07} {\sc S. Novo, R. Obaya, A.M. Sanz}, Exponential stability in
  non-autonomous delayed equations with applications to neural networks, {\em
    Discrete Contin. Dyn. Syst.} \textbf{18} (2007), 517--536.
\bibitem{noos13} {\sc S. Novo, R. Obaya, A.M. Sanz}, Uniform persistence and
  upper Lyapunov exponents for monotone skew-product semiflows, {\em
    Nonlinearity} \textbf{26} (2013), 2409--2440.
\bibitem{nosv} {\sc S. Novo, R. Obaya, A.M. Sanz, V.M. Villarragut}, The
  exponential ordering for non-autonomous delay systems with applications to
  compartmental Nicholson systems, to be published in {\em Proceedings of the
    Royal Society of Edinburgh Section A: Mathematics} (2023).

\bibitem{nuos4} {\sc C. N\'{u}\~{n}ez, R. Obaya, A.M. Sanz}, Minimal sets in
  monotone and concave skew-product semiflows I: A general theory, {\it
    J. Differential Equations\/} \textbf{252} (10) (2012), 5492--5517.
\bibitem{obsa16} {\sc R. Obaya, A.M. Sanz}, Uniform and strict persistence in
  monotone skew-product semiflows with applications to non-autonomous Nicholson
  systems, {\it J. Differential Equations\/} \textbf{261} (2016), 4135--4163.
\bibitem{obsa18} {\sc R. Obaya, A.M. Sanz}, Is uniform persistence a robust
  property in almost periodic models? A well-behaved family: almost-periodic
  Nicholson systems, {\em Nonlinearity} \textbf{31} (2018), 388--413.
\bibitem{shth} {\sc L.F. Shampine, S. Thompson}, Solving DDEs in MATLAB, {\em
    Applied Numerical Mathematics\/}, {\bf 37} (2001), 441--458.

\bibitem{shyi} {\sc W. Shen, Y. Yi}, Almost Automorphic and Almost Periodic
  Dynamics in Skew-Product Semiflows, {\em Mem. Amer. Math. Soc.}  {\bf 647},
  Amer. Math. Soc., Providence 1998.
\bibitem{smit} {\sc H.L. Smith}, {\em Monotone Dynamical Systems. An
    Introduction to the Theory of Competitive and Cooperative Systems},
  Amer. Math. Soc., Providence, 1995.
\bibitem{zhetal} {\sc H. Zhang, Q. Cao, H. Yang}, Asymptotically almost periodic
  dynamics on delayed Nicholson-type system involving patch structure, {\em
    J. Inequal. Appl.\/} (2020), doi.org/10.1186/s13660-020-02366-0.

\end{thebibliography}
\end{document}